\documentclass[11pt]{amsart}

\usepackage{amsopn}
\usepackage{amssymb, amscd}

\newcommand{\nc}{\newcommand}

\nc{\fg}{\mathfrak{f} } \nc{\vg}{\mathfrak{v} } \nc{\wg}{\mathfrak{w} }
\nc{\zg}{\mathfrak{z} } \nc{\ngo}{\mathfrak{n} } \nc{\kg}{\mathfrak{k} }
\nc{\mg}{\mathfrak{m} } \nc{\bg}{\mathfrak{b} } \nc{\ggo}{\mathfrak{g} }
\nc{\ggob}{\overline{\mathfrak{g}} } \nc{\sog}{\mathfrak{so} }
\nc{\sug}{\mathfrak{su} } \nc{\spg}{\mathfrak{sp} } \nc{\slg}{\mathfrak{sl} }
\nc{\glg}{\mathfrak{gl} } \nc{\cg}{\mathfrak{c} } \nc{\rg}{\mathfrak{r} }
\nc{\hg}{\mathfrak{h} } \nc{\tg}{\mathfrak{t} } \nc{\ug}{\mathfrak{u} }
\nc{\dg}{\mathfrak{d} } \nc{\ag}{\mathfrak{a} } \nc{\pg}{\mathfrak{p} }
\nc{\sg}{\mathfrak{s} } \nc{\affg}{\mathfrak{aff} }

\nc{\pca}{\mathcal{P}} \nc{\nca}{\mathcal{N}} \nc{\lca}{\mathcal{L}}
\nc{\oca}{\mathcal{O}} \nc{\mca}{\mathcal{M}} \nc{\tca}{\mathcal{T}}
\nc{\aca}{\mathcal{A}} \nc{\cca}{\mathcal{C}} \nc{\gca}{\mathcal{G}}
\nc{\sca}{\mathcal{S}} \nc{\hca}{\mathcal{H}} \nc{\bca}{\mathcal{B}}
\nc{\dca}{\mathcal{D}} \nc{\val}{\operatorname{val}}

\nc{\vp}{\varphi} \nc{\ddt}{\tfrac{{\rm d}}{{\rm d}t}}
\nc{\dpar}{\tfrac{\partial}{\partial t}} \nc{\im}{\mathtt{i}}

\nc{\SO}{\mathrm{SO}} \nc{\Spe}{\mathrm{Sp}} \nc{\Sl}{\mathrm{SL}}
\nc{\SU}{\mathrm{SU}} \nc{\Or}{\mathrm{O}} \nc{\U}{\mathrm{U}} \nc{\Gl}{\mathrm{GL}}
\nc{\Se}{\mathrm{S}} \nc{\Cl}{\mathrm{Cl}} \nc{\Spein}{\mathrm{Spin}}
\nc{\Pin}{\mathrm{Pin}} \nc{\G}{\mathrm{GL}_n(\RR)} \nc{\g}{\mathfrak{gl}_n(\RR)}

\nc{\RR}{{\Bbb R}} \nc{\HH}{{\Bbb H}} \nc{\CC}{{\Bbb C}} \nc{\ZZ}{{\Bbb Z}}
\nc{\FF}{{\Bbb F}} \nc{\NN}{{\Bbb N}} \nc{\QQ}{{\Bbb Q}} \nc{\PP}{{\Bbb P}}

\nc{\vs}{\vspace{.2cm}} \nc{\vsp}{\vspace{1cm}} \nc{\ip}{\langle\cdot,\cdot\rangle}
\nc{\ipp}{(\cdot,\cdot)} \nc{\la}{\langle} \nc{\ra}{\rangle} \nc{\unm}{\tfrac{1}{2}}
\nc{\unc}{\tfrac{1}{4}} \nc{\und}{\tfrac{1}{16}} \nc{\no}{\vs\noindent}
\nc{\lam}{\Lambda^2(\RR^n)^*\otimes\RR^n} \nc{\tangz}{{\rm T}^{\rm Zar}}
\nc{\nor}{{\sf n}}  \nc{\mum}{/\!\!/} \nc{\kir}{/\!\!/\!\!/}
\nc{\Ri}{\tfrac{4\Ric_{\mu}}{||\mu||^2}} \nc{\ds}{\displaystyle}
\nc{\ben}{\begin{enumerate}} \nc{\een}{\end{enumerate}} \nc{\f}{\frac}
\nc{\lb}{[\cdot,\cdot]} \nc{\isn}{\tfrac{1}{||v||^2}}
\nc{\gkp}{(\ggo=\kg\oplus\pg,\ip)} \nc{\ukh}{(\ug=\kg\oplus\hg,\ip)}
\nc{\tgkp}{(\tilde{\ggo}=\kg\oplus\pg,\ip)}
\nc{\wt}{\widetilde} \nc{\mm}{M}

\nc{\Hess}{\operatorname{Hess}} \nc{\ad}{\operatorname{ad}}
\nc{\Ad}{\operatorname{Ad}} \nc{\rank}{\operatorname{rank}}
\nc{\Irr}{\operatorname{Irr}} \nc{\End}{\operatorname{End}}
\nc{\Aut}{\operatorname{Aut}} \nc{\Inn}{\operatorname{Inn}}
\nc{\Der}{\operatorname{Der}} \nc{\Ker}{\operatorname{Ker}}
\nc{\Iso}{\operatorname{I}} \nc{\Diff}{\operatorname{Diff}}
\nc{\Lie}{\operatorname{L}} \nc{\tr}{\operatorname{tr}} \nc{\dif}{\operatorname{d}}
\nc{\sen}{\operatorname{sen}} \nc{\modu}{\operatorname{mod}}
\nc{\Ric}{\operatorname{R}} \nc{\Ricci}{\operatorname{Ric}}
\nc{\sym}{\operatorname{sym}} \nc{\symac}{\operatorname{sym^{ac}}}
\nc{\symc}{\operatorname{sym^{c}}} \nc{\scalar}{\operatorname{sc}}
\nc{\grad}{\operatorname{grad}} \nc{\ricci}{\operatorname{Rc}}
\nc{\Nor}{\operatorname{Norm}} \nc{\riccic}{\operatorname{ric^{c}}}
\nc{\riccig}{\operatorname{ric^{\gamma}}} \nc{\Rin}{\operatorname{M}}
\nc{\Le}{\operatorname{L}} \nc{\tang}{\operatorname{T}}
\nc{\level}{\operatorname{level}} \nc{\rad}{\operatorname{r}}
\nc{\abel}{\operatorname{ab}} \nc{\CH}{\operatorname{CH}}
\nc{\mcc}{\operatorname{mcc}} \nc{\Adj}{\operatorname{Adj}}
\nc{\Order}{\operatorname{O}}  \nc{\inj}{\operatorname{inj}} \nc{\proy}{\operatorname{pr}}
\nc{\vol}{\operatorname{vol}}

\theoremstyle{plain}
\newtheorem{theorem}{Theorem}[section]
\newtheorem{proposition}[theorem]{Proposition}
\newtheorem{corollary}[theorem]{Corollary}
\newtheorem{lemma}[theorem]{Lemma}

\theoremstyle{definition}
\newtheorem{definition}[theorem]{Definition}

\theoremstyle{remark}
\newtheorem{remark}[theorem]{Remark}

\title[Homogeneous Ricci solitons]{Structure of homogeneous Ricci solitons and the Alekseevskii conjecture}

\author{Ramiro Lafuente} \author{Jorge Lauret}

\address{Universidad Nacional de C\'ordoba, FaMAF and CIEM, 5000 C\'ordoba, Argentina}
\email{rlafuente@famaf.unc.edu.ar, lauret@famaf.unc.edu.ar}

\thanks{This research was partially supported by grants from CONICET, FONCYT and SeCyT UNC}

\begin{document}

\maketitle

\begin{abstract}
We bring new insights into the long-standing Alekseevskii conjecture, namely that any connected homogeneous Einstein manifold of negative scalar curvature is diffeomorphic to a Euclidean space, by proving structural results which are actually valid for any homogeneous expanding Ricci soliton, and generalize many well-known results on Einstein solvmanifolds, solvsolitons and nilsolitons.  We obtain that any homogeneous expanding Ricci soliton $M=G/K$ is diffeomorphic to a product $U/K\times N$, where $U$ is a maximal reductive Lie subgroup of $G$ and $N$ is the maximal nilpotent normal subgroup of $G$, such that the metric restricted to $N$ is a nilsoliton.  Moreover, strong compatibility conditions between the metric and the action of $U$ on $N$ by conjugation must hold, including a nice formula for the Ricci operator of the metric restricted to $U/K$.  Our main tools come from geometric invariant theory.  As an application, we give many Lie theoretical characterizations of algebraic solitons, as well as a proof of the fact that the following a priori much stronger result is actually equivalent to Alekseevskii's conjecture: Any expanding algebraic soliton is diffeomorphic to a Euclidean space.
\end{abstract}

\tableofcontents

\section{Introduction}\label{intro}

A major open question on homogeneous Einstein manifolds is known as Alekseevskii's conjecture (see \cite[7.57]{Bss}); namely

\begin{quote}
Any connected homogeneous Einstein manifold of negative scalar curvature is diffeomorphic to a Euclidean space.
\end{quote}

Known examples so far are all isometric to a left-invariant metric on a simply connected solvable Lie group.  This is also the situation for the broader class of homogeneous expanding Ricci solitons, i.e. $\ricci(g)=cg+\lca_Xg$ for some $c<0$ and vector field $X$.  Such examples are called {\it solvsolitons} in the literature ({\it nilsolitons} in the nilpotent case) and are defined by having their Ricci operators equal to a scalar multiple of the identity plus a derivation of the corresponding solvable Lie algebra.  Solvsolitons may be viewed as a rare case in Ricci soliton theory, as they are neither compact, nor gradient or K\"ahler.

In addition to their definition as a natural generalization of Einstein metrics, Ricci solitons correspond to solutions of the Ricci flow that evolve self similarly, that is, only by scaling and pullback by diffeomorphisms, and they therefore play an important role in the singularity behavior of the Ricci flow on any class of manifolds.

As a first step toward a possible resolution of the above conjecture, we obtain in the present paper some structural results which are actually valid for any homogeneous expanding Ricci soliton, and generalize the results on Einstein solvmanifolds obtained in \cite{Hbr,standard}, as well as those on solvsolitons given in \cite{solvsolitons}.  As in these works, the technology we use here is derived from strong results in geometric invariant theory.  The approach is developed in Section \ref{pre} in the general setting of homogeneous spaces, independently of the rest of the paper, since we believe it might be useful to study for instance other questions on Ricci curvature of homogeneous manifolds.

Our main result can be described as follows (see Theorem \ref{main} and Section \ref{cons} for more detailed statements).

\begin{theorem}
For any simply connected homogeneous expanding Ricci soliton $(M,g)$, say with $\ricci(g)=cg+\lca_Xg$, there exist a presentation $(M,g)=G/K$ as an almost-effective homogeneous space and a reductive decomposition $\ggo=\kg\oplus\pg$ such that the following conditions hold:

\begin{itemize}
\item $\ggo=\ug\oplus\ngo$ is a semi-direct product of Lie algebras, with $\ug$ reductive and $\ngo$ nilpotent, such that $\hg\perp\ngo$ with respect to the inner product on $\pg$ determined by $g$, where
$$
    \ggo= \rlap{$\overbrace{\phantom{\kg\oplus\hg}}^\ug$} \kg \oplus \underbrace{\hg\oplus\ngo}_\pg.
$$

\item $G\simeq U\ltimes N$ and as a differentiable manifold,
$$
M=U/K\times N,
$$
where $U$ and $N$ are the respective connected Lie subgroups of $G$ with Lie algebras $\ug$ and $\ngo$.

\item The metric restricted to $N$ is a nilsoliton, with Ricci operator $\Ricci_N=cI+D_1$ for some $D_1\in\Der(\ngo)$.

\item The Ricci operator of the metric restricted to $U/K$, relative to the reductive decomposition $\ug=\kg\oplus\hg$, is given by $\Ricci_{U/K}=cI+C_{\hg}$, where $C_{\hg}$ is the positive semi-definite operator of $\hg$ defined by
$$
\la C_{\hg}Y,Y\ra=\unc\tr{\left(\ad{Y}|_\ngo+(\ad{Y}|_\ngo)^t\right)^2}, \qquad\forall Y\in\hg.
$$

\item The adjoint action of $\ug$ on $\ngo$ and the metric satisfy the following extra compatibility condition:
$$
\sum [\ad{Y_i}|_\ngo,(\ad{Y_i}|_\ngo)^t]=0, \qquad\mbox{for any orthonormal basis $\{ Y_i\}$ of $\hg$},
$$
from which it follows that $(\ad{Y}|_\ngo)^t\in\Der(\ngo)$ for all $Y\in\ug$.
\end{itemize}

Conversely, if these conditions hold, then $G/K$ is an expanding Ricci soliton with Ricci operator
\begin{equation}\label{sa-intro}
\Ricci=cI+\unm(D_\pg+D_\pg^t),
\end{equation}
where
$$
D=\left[\begin{smallmatrix} 0&0\\ 0&D_\pg\\\end{smallmatrix}\right] = -\ad{H}+ \left[\begin{smallmatrix} 0&&\\ &0&\\ &&D_1\\\end{smallmatrix}\right]\in\Der(\ggo),
$$
and $H\in\pg$ is defined by $\la H,X\ra=\tr{\ad{X}}$ for all $X\in\pg$.
\end{theorem}

\begin{remark}
Regarding Alekseevskii's conjecture, we note that $M$ is diffeomorphic to a Euclidean space if and only if $U/K$ is so, which is equivalent to $K$ being a maximal compact subgroup of the reductive Lie group $U$.
\end{remark}

Condition (\ref{sa-intro}) implies that the Ricci tensor equals $\ricci(g)=cg+\lca_{X_D}g$, where $X_D$ is the vector field  defined by the one-parameter subgroup of equivariant diffeomorphisms $\vp_t\in\Diff(G/K)$ determined by the automorphisms $e^{tD}\in\Aut(\ggo)$, and the Ricci flow starting at $g$ evolves by $g(t)=(-2ct+1)\vp_{s(t)}^*g$, $\; t\in(\tfrac{1}{2c},\infty)$, where $s(t)=\log(-2ct+1)/c$.  These homogeneous spaces are called {\it semi-algebraic solitons}, and actually any homogeneous Ricci soliton is semi-algebraic with respect to its full isometry group $G=\Iso(M,g)$ (see \cite{Jbl}).  If the derivation $D$ above is symmetric and hence $\Ricci=cI+D_\pg$, then $G/K$ is said to be an {\it algebraic soliton}, which is precisely the algebro-geometric notion generalizing to any homogeneous space the definition of a solvsoliton (see \cite[Section 3]{homRS} for a recent overview on homogeneous Ricci solitons).

Geometrically speaking, algebraic solitons are characterized among homogeneous Ricci solitons as those for which the Ricci flow solution is simultaneously diagonalizable (see \cite{homRS}).  As a first application of our structural results, we give many Lie theoretical characterizations of algebraic solitons, as for instance $\ad{H}$ being a normal operator, or also $\Ricci|_{\hg}=cI$ (see Proposition \ref{equiv-algsol}).

In Section \ref{algsol-alek}, we use the above theorem to prove that given any expanding algebraic soliton $(M,g)=G/K$, with $G$ non-unimodular, one can modify the metric $g$ in order to obtain a new homogeneous metric $\tilde{g}$ on $M$ which is Einstein with $\ricci(\tilde{g})=c\tilde{g}$.  In the unimodular case the Einstein metric is obtained on $\RR\times M$.  In any case, this implies that the following a priori much stronger result is equivalent to the Alekseevskii conjecture:

\begin{quote}
Any expanding algebraic soliton $G/K$ is diffeomorphic to a Euclidean space (or equivalently, $K$ is a maximal compact subgroup of $G$).
\end{quote}

\begin{remark}
The statement inside the parentheses is actually the conclusion of Alekseevskii's conjecture as originally stated in Besse's book \cite[7.57]{Bss} (see Section \ref{ERS} for an account of the conjecture and its equivalent formulations).  It is worth noticing that if $G$ is a linear group, then the maximality of $K$ implies that $G/K$ is isometric to a left-invariant metric on a simply connected solvable Lie group (see \cite[Corollary 1,c]{Wlf}), i.e. a solvsoliton.
\end{remark}

We obtain in addition a generalization of the link between Einstein solvmanifolds and nilsolitons discovered in \cite{soliton}, to the case of homogeneous Einstein manifolds $G/K$ of negative scalar curvature and unimodular expanding algebraic solitons $G_0/K$, which as differentiable manifolds satisfy $G/K=\RR\times G_0/K$ (see Proposition \ref{Et}).

\section{Algebraic aspects of homogeneous Riemannian manifolds}\label{pre}

We consider in this section an `algebraic' point of view to study homogeneous
Riemannian manifolds, with special attention on the noncompact case, which allows us
to put to a good use all the results from geometric invariant theory described in Appendix \ref{git}.

A Riemannian manifold $(M,g)$ is said to be {\it homogeneous} if its isometry group
$\Iso(M,g)$ acts transitively on $M$.  A {\it homogeneous (Riemannian) space} is instead a
differentiable manifold $G/K$, where $G$ is a Lie group and $K\subset G$ is a closed
Lie subgroup, endowed with a $G$-invariant Riemannian metric.  Both concepts are of
course intimately related, though not in a one-to-one way.  To study a
geometric problem on homogeneous manifolds it is often very useful and
healthy to capture the relevant algebraic information and present the hypotheses and
the problem in `algebraic' terms.

Given a connected homogeneous manifold $(M,g)$, each transitive closed Lie subgroup $G\subset\Iso(M,g)$ gives rise to a
presentation of $(M,g)$ as a homogeneous space $(G/K,g)$, where
$K$ is the isotropy subgroup of $G$ at some point $o\in M$.  As $K$ is compact, there always
exists a {\it reductive} (i.e. $\Ad(K)$-invariant) decomposition $\ggo=\kg\oplus\pg$, where $\ggo$ and $\kg$ are respectively the Lie algebras of $G$ and $K$, and thus $\pg$ can be naturally identified with the tangent space $\pg\equiv T_oM=T_oG/K$, by taking the value at the origin $o=eK$ of the Killing vector fields
corresponding to elements of $\pg$ (i.e. $X_o=\ddt|_0\exp{tX}(o)$).  Any $G$-invariant metric $g$ on $G/K$ can therefore be identified with $\ip:=g(o)$, an $\Ad(K)$-invariant inner product on $\pg$.

We can extend $\ip$ to an $\Ad(K)$-invariant inner product on $\ggo$ such that $\la\kg,\pg\ra=0$.  Let us call the data set $\gkp$ a {\it metric reductive decomposition}.

In order to get a presentation $(M,g)=(G/K,g)$ of a connected homogeneous manifold as a homogeneous space, there is no need for $G\subset\Iso(M,g)$ to hold, that is, to have an {\it effective} action.  It is actually enough to have a transitive action of $G$ on $M$ by isometries that is {\it almost-effective} (i.e. the normal subgroup $\{ g\in G:ghK=hK, \;\forall h\in G\}$ of $K$ is discrete), along with a reductive decomposition $\ggo=\kg\oplus\pg$ such that the inner product $\ip$ on $\pg$ defined by $\ip:=g(o)$ is $\Ad(K)$-invariant.  In this more general case, $K$ might not be compact, but it is easy to prove that $\overline{\Ad(K)}$ is still compact in $\Gl(\ggo)$ and hence a metric reductive decomposition $\gkp$ can still be attached to $(G/K,g)$.

Moreover, since the Killing form $B$ of $\ggo$ is always negative definite on $\kg$ (recall that the isotropy representation $\ad:\kg\longrightarrow\End(\pg)$ is faithful by almost-effectiveness), there is always a distinguished reductive decomposition to choose; namely, the one where $\pg$ is precisely the orthogonal complement of $\kg$ with respect to $B$ (see e.g. \cite[Lemma 2.1]{homRS}).

Summarizing, a given homogeneous manifold $(M,g)$ can always be presented as a homogeneous space $(G/K,g)$ with a metric reductive decomposition $\gkp$, and it can always be assumed that $B(\kg,\pg)=0$.  Any homogeneous space considered in this paper will be assumed to be almost-effective and connected, unless otherwise stated.

\begin{remark} When $\kg=0$, a metric reductive decomposition consists only of the pair $(\ggo,\ip)$, which is
often referred to in the literature as a {\it metric Lie algebra} and corresponds to a left-invariant metric on a Lie group.
\end{remark}

Let us fix a homogeneous space $(G/K,g)$ together with a metric reductive decomposition $\gkp$ for the rest of the section.  We consider
\begin{equation}\label{decp}
\pg=\hg\oplus\ngo,
\end{equation}
the orthogonal decomposition with respect to $\ip$, where $\ngo$ is the {\it nilradical} of $\ggo$ (i.e. the maximal nilpotent ideal; recall that $\ngo\subset\pg$ as it is contained in the kernel of $B$).  Thus the $\pg$-component
$$
\lb_{\pg}:\pg\times\pg\longrightarrow\pg
$$
of the Lie bracket $\lb$ of $\ggo$ restricted to $\pg\times\pg$ can be decomposed as a sum of bilinear maps as follows:
$$
[\cdot,\cdot]_{\pg}=\lambda_0+\lambda_1+\eta+\mu,
$$
where
\begin{equation}\label{dec}
\begin{array}{lr}
\lambda_0:\hg\times\hg\longrightarrow\hg, &\quad \eta:\hg\times\ngo\longrightarrow\ngo,\\
\lambda_1:\hg\times\hg\longrightarrow\ngo,  &\quad \mu:\ngo\times\ngo\longrightarrow\ngo. \\
\end{array}
\end{equation}

We are using here that $\ngo$ is an ideal of $\ggo$.  Also note that $\lambda_0,\lambda_1$ and $\mu$ are skew-symmetric, $(\ngo,\mu)$ is a nilpotent Lie algebra, and by convention $[X,Y]=-\eta(Y,X)$ for all $Y\in\hg$, $X\in\ngo$.

In order to apply the results in Appendix \ref{git}, if $\dim{\ngo}=n$, then we identify $\ngo$ with $\RR^n$ via an orthonormal basis $\{ e_1,...,e_n\}$ of $\ngo$. In this way, $\mu$ can be viewed as a point in the variety of nilpotent Lie algebras $\nca\subset V$. If $\mu\ne 0$, then $\mu$ belongs to some stratum $\sca_{\beta}$ with $\beta\in\tg^+$ (see Theorem \ref{strata}), and we can define $E_{\beta}\in\End(\pg)$ by
\begin{equation}\label{ebeta}
E_{\beta}:=\left[\begin{smallmatrix} 0&\\ &\beta+||\beta||^2I\end{smallmatrix}\right], \qquad \mbox{i.e.} \quad
E_{\beta}|_{\hg}=0, \quad E_{\beta}|_{\ngo}=\beta+||\beta||^2I.
\end{equation}

\begin{remark}
It is still unclear what would be the geometric meaning of this diagonal $n\times n$-matrix $\beta$
associated to each metric reductive decomposition, and consequently, to any homogeneous manifold.
However, $\beta$ will play a fundamental role in the proofs of most of the
structural results on homogeneous Ricci solitons obtained in the present paper.
\end{remark}

On the other hand, we will also identify $[\cdot,\cdot]_{\pg}$ with an element of $\Lambda^2\pg^*\otimes\pg$ and use the notation introduced in Appendix \ref{git} by replacing $\RR^n$ with $\pg$.

The following technical result will be crucial in the applications.

\begin{lemma}\label{pie}
If $\mu\in\sca_{\beta}$ satisfies $\beta_{\mu}=\beta$, then
$$
\la\pi(E_{\beta})[\cdot,\cdot]_{\pg},[\cdot,\cdot]_{\pg}\ra\geq 0.
$$
\end{lemma}

\begin{proof}
We first note that $\pi(E_{\beta})$ leaves invariant the subspaces of $\Lambda^2\pg^*\otimes\pg$ the $\lambda_i$'s, $\eta$ and $\mu$ belong to (for instance $\lambda_0\in\Lambda^2\hg^*\otimes\hg$, $\lambda_1\in\Lambda^2\hg^*\otimes\ngo$, etc).  These subspaces are orthogonal with
respect to the inner product $\ip$ defined in (\ref{innV}) and hence
\begin{align*}
\la\pi(E_{\beta})[\cdot,\cdot]_{\pg},[\cdot,\cdot]_{\pg}\ra=& \la\pi(E_{\beta})\lambda_0,\lambda_0\ra + \la\pi(E_{\beta})\lambda_1,\lambda_1\ra \\
& + \la\pi(E_{\beta})\eta,\eta\ra + \la\pi(E_{\beta})\mu,\mu\ra.
\end{align*}
Let $\{ Y_i\}$ and $\{ X_k\}$ denote orthonormal bases of $\hg$ and $\ngo$, respectively.  It is easy to check that
$$
\la\pi(E_{\beta})\lambda_0,\lambda_0\ra=0,
$$
and since $E_{\beta}|_{\ngo}$ is positive definite by (\ref{betapos}), we have that
\begin{equation}\label{l2l3e1}
\la\pi(E_{\beta})\lambda_1,\lambda_1\ra =\sum\la E_{\beta}\lambda_1(Y_i,Y_j),\lambda_1(Y_i,Y_j)\ra\geq 0.
\end{equation}

If $\ad_{\eta}{Y}$ denote the derivation of $(\ngo,\mu)$ defined by the adjoint action of $Y\in\hg$ on $\ngo$, then
\begin{align}
\la\pi(E_{\beta})\eta,\eta\ra  =& 2\sum\la E_{\beta}\eta(Y_i,X_j)-\eta(Y_i,E_{\beta}X_j),\eta(Y_i,X_j)\ra \notag \\
 =& 2\sum\la[\beta,\ad_{\eta}{Y_i}](X_j),\ad_{\eta}{Y_i}(X_j)\ra \label{l4} \\
 =& 2\sum\la[\beta,\ad_{\eta}{Y_i}],\ad_{\eta}{Y_i}\ra. \notag
\end{align}
It follows from (\ref{adbeta}) that $\la\pi(E_{\beta})\eta,\eta\ra\geq 0$. Finally, we have that
\begin{equation}\label{mu}
\la\pi(E_{\beta})\mu,\mu\ra
=\left\la\pi\left(\beta+||\beta||^2I\right)\mu,\mu\right\ra\geq 0
\end{equation}
by (\ref{delta}), which concludes the proof of the lemma.
\end{proof}

There is a unique element $H\in\ggo$ which satisfies
\begin{equation}\label{defH}
\la H,X\ra=\tr{\ad{X}}, \qquad\forall X\in\ggo.
\end{equation}
We note that $H\in\pg$ due to the $\ad{\kg}$-invariance of $\ip$, and that $H=0$ if and only if $\ggo$ is unimodular. Let $B_{\ggo}:\ggo\longrightarrow\ggo$ denote the symmetric map defined by the Killing form of $\ggo$ relative to $\ip$, that is,
\begin{equation}\label{defKilling}
\la B_{\ggo}X,Y\ra=\tr{\ad{X}\ad{Y}},\qquad\forall X,Y\in\ggo.
\end{equation}
It follows that, relative to the decomposition $\ggo = \kg \oplus \pg$, the operator $B_\ggo$ has the form
\begin{equation}\label{Bg}
B_\ggo = \left[\begin{smallmatrix} B_0&\ast\\ \ast&B_\pg\\\end{smallmatrix}\right], \qquad B_0<0,
\end{equation}
with $\ast=0$ if one is assuming that $B(\kg,\pg)=0$.

The {\it Ricci operator} $\Ricci$ of $(G/K,g)$ with metric reductive decomposition $\gkp$ is given by (see \cite[7.38]{Bss}),
\begin{equation}\label{ricci}
\Ricci=\mm-\unm B_{\pg}-S(\ad_{\pg}{H}),
\end{equation}
where $\ad_{\pg}{H}:\pg\longrightarrow\pg$ is defined by
$\ad_{\pg}{H}(X)=[H,X]_{\pg}$ for all $X\in\pg$,
\begin{equation}\label{sym}
S(A):=\unm(A+A^t),
\end{equation}
is the symmetric part of an operator, and $\mm:\pg\longrightarrow\pg$ is the symmetric operator defined by
\begin{align}
\la \mm X,Y\ra =& -\unm\sum\la [X,X_i]_{\pg},X_j\ra\la [Y,X_i]_{\pg},X_j\ra \label{R} \\
 &+ \unc\sum\la [X_i,X_j]_{\pg},X\ra\la [X_i,X_j]_{\pg},Y\ra, \qquad\forall
X,Y\in\pg, \notag
\end{align}
where $\{ X_i\}$ is any orthonormal basis of $(\pg,\ip)$.

It follows from (\ref{mmv}) that this `anonymous' map $\mm$ in the formula of the Ricci operator satisfies
\begin{equation}\label{mmR}
m([\cdot,\cdot]_{\pg})=\tfrac{4}{||[\cdot,\cdot]_{\pg}||^2}\mm,
\end{equation}
where $m:\Lambda^2\pg^*\otimes\pg\longrightarrow\sym(\pg)$ is the moment map for the natural action of $\Gl(\pg)$ on $\Lambda^2\pg^*\otimes\pg$.  In other words, $\mm$ may be alternatively defined by
\begin{equation}\label{Rmm}
\tr{\mm E}=\unc\la \pi(E)[\cdot,\cdot]_{\pg},[\cdot,\cdot]_{\pg}\ra,
\qquad\forall E\in\End(\pg),
\end{equation}
where we are considering $[\cdot,\cdot]_{\pg}$ as a vector in $\Lambda^2\pg^*\otimes\pg$, $\ip$ is the inner product defined in (\ref{innV}) and $\pi$ is the representation given in (\ref{actiong}) (see the notation in Appendix \ref{git} and replace the vector space $\RR^n$ by $\pg$).

\begin{remark}\label{Rort}
In particular, $\mm$ is orthogonal to any derivation of the algebra
$(\pg,\lb_{\pg})$.
\end{remark}

\begin{remark}\label{Mcomm}
It also follows that the moment map $M$ commutes with any derivation of the algebra $(\pg,\lb_{\pg})$ whose transpose is also a derivation.  Indeed, if $E$ is such a derivation, then
\begin{align*}
\tr{[\mm,E] C} =& \tr{\mm [E,C]} = \unc \la \pi([E,C])\lb_\pg,\lb_\pg \ra \\
=& \unc \la\pi(C)\lb_\pg,\pi(E^t)\lb_\pg\ra - \la\pi(C)\pi(E)\lb_\pg,\lb_\pg\ra = 0,
\end{align*}
for all $C\in \End(\pg)$.
\end{remark}

Notice that the above two remarks are actually valid for any algebra and its moment map value as considered in the appendix.

We now prove some technical results involving $H$, $\Ricci$ and the adjoint representation of $\kg$ on $\pg$, which will be useful in the following sections.

Let $\pg=\hg\oplus\ngo$ be the orthogonal decomposition defined in (\ref{decp}), and consider $\lb_{\pg}=\lambda_0+...+\mu$ as in {\rm (\ref{dec})}. Since $\ngo$ is the nilradical of $\ggo$ we have that $H\perp\ngo$ and $\ngo$ is in the kernel of the Killing form.  Thus,
$$
H\in\hg,
$$
and $\ad_{\pg}{H}$ and $B_{\pg}$ have the form
\begin{equation}\label{adHB}
\ad_{\pg}{H}=\left[\begin{smallmatrix} \ad_{\lambda_0}{H}&0\\  \ad_{\lambda_1} H &\ad_\eta H \end{smallmatrix}\right], \qquad
B_{\pg}=\left[\begin{smallmatrix} B_1&0\\ 0&0\end{smallmatrix}\right],
\end{equation}
where $\ad_{\lambda_i}{Y}(Y')=\lambda_i(Y,Y')$ for all $Y,Y'\in\hg$, and analogously for $\eta$.

We recall that
$$
\ggo=\kg\oplus\hg\oplus\ngo, \quad [\kg,\kg]\subset\kg, \quad
[\kg,\hg]\subset\hg, \quad [\ggo,\ngo]\subset\ngo,
$$
and thus the Lie bracket of $\ggo$ decomposes as
$$
\lb=\nu_0+\nu_1+\nu_2+\lambda_2+\lb_{\pg},
$$
where
$$
\begin{array}{lll}
\nu_0:\kg\times\kg\longrightarrow\kg, && \lambda_2:\hg\times\hg\longrightarrow\kg. \\
\nu_1:\kg\times\hg\longrightarrow\hg, && \\
\nu_2:\kg\times\ngo\longrightarrow\ngo, &&
\end{array}
$$
In all of that follows, it will be very useful to have in mind the following formulas for the adjoint operators relative to the decomposition
$\ggo=\kg\oplus\hg\oplus\ngo$.  For $Z\in\kg$, $Y\in\hg$, and
$X\in\ngo$ we have,
$$
\begin{array}{c}
\ad{Z}=\left[\begin{smallmatrix}
\ad_{\nu_0}{Z}&0&0\\
0&\ad_{\nu_1}{Z}&0\\
0&0&\ad_{\nu_2}{Z}
\end{smallmatrix}\right], \qquad
\ad{Y}=\left[\begin{smallmatrix}
0&\ad_{\lambda_2}{Y}&0\\
\ad_{\nu_1}{Y}&\ad_{\lambda_0}{Y}&0\\
0&\ad_{\lambda_1}{Y}&\ad_{\eta}{Y}
\end{smallmatrix}\right], \\
\ad{X}=\left[\begin{smallmatrix}
0&0&0\\
0&0&0\\
0&0&0\\
\ad_{\nu_2}{X}&\ad_{\eta}{X}&\ad_{\mu}{X}
\end{smallmatrix}\right].
\end{array}
$$
If $\ug:=\kg\oplus\hg$ then it is easy to see that $\lb_{\ug}:=\nu_0+\nu_1+\lambda_5+\lambda_0$ satisfies the Jacobi condition; moreover, $(\ug,\lb_{\ug})$ is a {\it reductive} (i.e. semisimple plus a center) Lie algebra isomorphic to $\ggo/\ngo$. This implies that $\tr{\ad_{\lambda_0}}{Y}=\tr{\ad_{\ug}{Y}}=0$, as every reductive Lie algebra is unimodular, and so
\begin{equation}\label{tradY}
\tr{\ad Y} = \la H,Y\ra=\tr{\ad_{\eta}{Y}}, \qquad \forall Y\in\hg.
\end{equation}
On the other hand, one obtains from the Jacobi identity for $\ggo$ that
$\ad_{\eta}{[Z,Y]}=[\ad_{\nu_2}{Z},\ad_{\eta}{Y}]$, from which it follows
that
$$
\la[Z,H],Y\ra=-\la H, [Z,Y]\ra=-\tr{\ad_{\eta}{[Z,Y]}}=0,
\qquad\forall Z\in\kg, Y\in\hg,
$$
and therefore,
\begin{equation}\label{kH0}
[\kg,H]=0.
\end{equation}
Since the Ricci tensor is $\ad{\kg}$-invariant one obtains that
\begin{equation}\label{Z}
[\ad{Z}|_{\pg},\Ricci]=0, \qquad\forall Z\in\kg.
\end{equation}

We conclude this section with a technical lemma about the derivations of $\ggo$ that leave $\kg$ invariant.

\begin{lemma}\label{Dpp}
If we consider the reductive decomposition with $B(\kg,\pg)=0$ and $D \in \Der(\ggo)$ is a derivation such that $D \kg \subseteq \kg$, then
$$
D\pg \subseteq \pg, \qquad D\ngo \subseteq \ngo, \qquad \tr{D|_\pg} = \tr{D|_\ngo}.
$$
\end{lemma}

\begin{proof}
From $D \kg \subseteq \kg$ we get that $D$ has the form
\begin{equation}\label{D}
D = \left[\begin{smallmatrix} D_\kg&C\\ 0&D_\pg\\\end{smallmatrix}\right].
\end{equation}
On the other hand, $e^{tD} \in \Aut(\ggo)$ and hence $e^{-tD^t} B_\ggo e^{-tD} = B_\ggo$ for all $t$, so differentiating at $t=0$ yields
\begin{equation}\label{DB}
B_\ggo D + D^t B_\ggo = 0.
\end{equation}
Now we use this formula along with the information from \eqref{Bg} and \eqref{D} to obtain that $B_0 C = 0$, and therefore $C=0$ since $B_0$ is negative definite. Thus $D\pg\subset\pg$, and since $\ngo$ is the nilradical of $\ggo$ we obtain that $D\ngo\subset\ngo$.

In order to prove the last assertion, we observe that since $\ngo \subseteq \Ker B_\ggo \subseteq \pg$, we may assume without any loss of generality that there is an orthogonal decomposition $\hg = \hg_1 \oplus \ag$ such that $\ag \oplus \ngo = \Ker B_\ggo$. Recall that $\Ker B_\ggo$ and $\ngo$ are both characteristic ideals of $\ggo$, so that every derivation leaves them invariant. Moreover, every derivation of $\ggo$ carries $\Ker B_\ggo$ into $\ngo$, as $\ngo$ is also the nilradical of the solvable Lie algebra $\Ker B_\ggo$ (see e.g. \cite[Lemma 2.6]{GrdWls}).  Then, relative to the decomposition $\ggo = \kg \oplus \hg_1 \oplus \ag \oplus \ngo$, we see that $B_\ggo$ and $D$ have the following forms:
\begin{equation}\label{formBD}
B_\ggo = \left[\begin{smallmatrix}
B_0&0&0&0\\
0&B_2&0&0\\
0&0&0&0\\
0&0&0&0
\end{smallmatrix}\right], \qquad
D = \left[\begin{smallmatrix}
D_\kg&0&0&0\\
0&D_1&0&0\\
0&\star&0&0\\
0&\star&\star&D_\ngo
\end{smallmatrix}\right].
\end{equation}
It now follows from \eqref{DB} that $B_2 D_1 + D_1^t B_2 = 0$, but since $B_2$ is invertible, we obtain that
$$
D_1^t = -B_2 D_1 B_2^{-1}.
$$
This implies that $\tr{D_1} = 0$, concluding the proof.
\end{proof}

\begin{remark}\label{Bort}
It follows from \eqref{DB} that $\tr{B_{\pg}D_{\pg}}=0$ for any $\left[\begin{smallmatrix} 0&0\\ 0&D_\pg\\\end{smallmatrix}\right]\in\Der(\ggo)$.  If in addition $B(\kg,\pg)=0$, then this holds for any  $\left[\begin{smallmatrix} 0&\ast\\ \ast&D_\pg\\\end{smallmatrix}\right]\in\Der(\ggo)$.
\end{remark}

\section{Einstein and Ricci soliton homogeneous manifolds}\label{ERS}

Let $M$ be a differentiable manifold.  A Riemannian metric $g$ on $M$ is called {\it
Einstein} if its Ricci tensor $\ricci(g)$ satisfies
$$
\ricci(g)=cg, \qquad \mbox{for some}\; c\in\RR.
$$
The scalar $c$ is sometimes called the {\it cosmological constant}.  A classical reference for Einstein manifolds is the book \cite{Bss}, and some
updated expository articles are \cite{And}, \cite{LbrWng}, \cite[III,C.]{Brg1} and
\cite[11.4]{Brg2}.

In the homogeneous
case, the Einstein equation becomes a subtle system of algebraic equations, and the
following main general question is still open in both, the compact and noncompact cases:

\begin{quote}
Which homogeneous spaces $G/K$ admit a $G$-invariant Einstein Riemannian metric?
\end{quote}

We refer to the surveys \cite{Wng, cruzchica} and the references therein for an update in
the compact and noncompact cases, respectively.

In the noncompact homogeneous case, the only known examples until now are all of a very particular kind; namely,
solvable Lie groups endowed with a left invariant metric (so called
{\it solvmanifolds}).  Moreover, every known example is simply connected and consequently diffeomorphic to a Euclidean space.  The following long standing conjecture has been attributed to Dmitrii Alekseevskii.

\begin{quote}
{\bf Alekseevskii's conjecture} \cite[7.57]{Bss}.  Let $(G/K,g)$ be a homogeneous
Riemannian space such that $G\subset\Iso(G/K,g)$ is a closed Lie subgroup.  If $(G/K,g)$ is Einstein of negative scalar curvature, then $K$ is a maximal compact subgroup of $G$.
\end{quote}

When $G$ is a linear group, maximality of $K$ implies that $(G/K,g)$ is isometric to a (simply connected) solvmanifold (see \cite[Corollary 1,c]{Wlf}). The conjecture is wide open, and it
is known to be true only for $\dim \leq 5$, a result which follows from the complete
classification in these dimensions given in \cite{Nkn2}.  In \cite{Nkn1}, many
examples of noncompact homogeneous spaces which do not admit an Einstein invariant
metric are given.  However, many of such examples do admit invariant metrics of negative Ricci
curvature.  We refer to \cite{DttLt} and \cite{DttLtMtl} for examples of left invariant metrics of
negative Ricci curvature on the Lie groups $\Sl_n(\RR)$, $n\geq 3$ and on any
complex simple Lie group, respectively.

\begin{proposition}\label{equivalek}
Let $(M,g)$ be a connected homogeneous Riemannian manifold.  Then the following
conditions are equivalent:
\begin{itemize}
\item[(i)] There is a closed connected Lie subgroup $G\subset \Iso(M,g)$ acting transitively on $M$ such that its isotropy $K$ is a maximal compact subgroup of $G$.

\item[(ii)] $M$ is diffeomorphic to a Euclidean space.

\item[(iii)] For any homogeneous Riemannian space $(G/K,g)$ with $G$ connected and $K$ compact which is isometric to $(M,g)$, $K$ is a maximal compact subgroup of $G$.
\end{itemize}
\end{proposition}

\begin{proof}
It is well known that if $G$ is any connected Lie group and $K\subset G$ is a compact Lie
subgroup, then $G/K$ is diffeomorphic to a Euclidean space if and only if $K$ is a
maximal compact subgroup of $G$ (see e.g. \cite{Iws}, \cite{Hch} or the recent book \cite[Sections 13.1-13.3]{HlgNeb}).  Thus the equivalence between part (ii) and any of the other two claims follows.
\end{proof}

The conjecture can therefore be rephrased in the following more transparent way:

\begin{quote}
{\bf Alekseevskii's conjecture}.  Any Einstein connected homogeneous Riemannian
manifold of negative scalar curvature is diffeomorphic to a Euclidean space.
\end{quote}

It also follows from Proposition \ref{equivalek} that to prove that the conjecture
holds for a given homogeneous manifold, one is allowed to
work with any of its presentations as a homogeneous space.

Among the several important contributions of Ricci flow theory, the following notion generalizing the Einstein condition has had a great impact in differential geometry.  A complete Riemannian manifold $(M,g)$ is said to be a {\it Ricci soliton} if
\begin{equation}\label{rseq}
\ricci(g)=cg+\lca_Xg,
\end{equation}
where $c\in\RR$ is the so called {\it cosmological constant}, as in the Einstein case, and $\lca_Xg$ is the usual Lie derivative of $g$ in the direction of a (complete) differentiable vector field $X$ on $M$.  If in addition $X$ is the gradient field of a smooth function on $M$
then $g$ is called a {\it gradient} Ricci soliton.  A metric $g$ is a Ricci
soliton if and only if the solution to the Ricci flow starting at $g$, given by
$$
\dpar g(t)=-2\ricci(g(t)), \qquad g(0)=g,
$$
evolves self similarly, that is,
\begin{equation}\label{rssol}
g(t)=c_t\vp_t^*g,
\end{equation}
for some scaling $c_t>0$ and one-parameter family $\vp_t$ of diffeomorphisms of $M$.  It is well-known that one can always assume that $c_t=-2ct+1$, where $c$ is the cosmological constant of the Ricci soliton $g$.  We refer to \cite{libro,Cao,DncHllWng} and the references therein for further information on Ricci solitons.

From results due to Ivey, Naber, Perelman and Petersen-Wylie (see \cite[Section 2]{solvsolitons}), it follows that any
{\it nontrivial} (i.e. non-Einstein and not the product of an Einstein homogeneous manifold with a Euclidean space) homogeneous Ricci soliton must be noncompact, {\it expanding} (i.e. $c<0$) and non-gradient.  As for the Einstein case, any known example so far of a nontrivial homogeneous Ricci soliton is isometric to a simply connected solvmanifold.

The following is a natural way to consider a homogeneous Ricci soliton `algebraic', in the sense that the algebraic structure of some of its presentations as a homogeneous space is strongly involved.

\begin{definition}\label{sas}\cite[Definition 1.4]{Jbl}
A homogeneous space $(G/K,g)$ is called a {\it semi-algebraic
soliton} if there exists a one-parameter family $\tilde{\vp}_t\in\Aut(G)$ with $\tilde{\vp}_t(K)=K$ such that
$$
g(t)=c_t\vp_t^*g, \qquad g(0)=g,
$$
is a solution to the Ricci flow for some scaling $c_t>0$, where $\vp_t\in\Diff(G/K)$ is the equivariant diffeomorphism determined by $\tilde{\vp}_t$.
\end{definition}

\begin{remark}\label{GscKc}
Let $(G/K,g)$ be a semi-algebraic soliton.  If $K_0$ is the connected component of $K$, then the cover $(G/K_0,g)$ is also a semi-algebraic soliton since the automorphisms satisfy $\tilde{\vp}_t(K_0)=K_0$ and the metrics are locally isometric.  Similarly, $(\wt{G}/\wt{K},g)$ is a semi-algebraic soliton, where $q:\wt{G}\longrightarrow G$ is the simply connected cover and $\wt{K}=q^{-1}(K)$, as each $\tilde{\vp}_t$ can be lift to an automorphism of $\wt{G}$ making the corresponding diagram commutative.  The simply connected cover $(\wt{G}/\wt{K}_0,g)$ of $(G/K,g)$ is therefore a semi-algebraic soliton as well.  In all the above cases, the metric is canonically defined and has been denoted by $g$.
\end{remark}

In terms of a metric reductive decomposition, semi-algebraic solitons are characterized as follows.

\begin{proposition}\cite{Jbl,homRS}\label{semi-car}
If a homogeneous space $(G/K,g)$ is a semi-algebraic soliton, then for the metric reductive decomposition $(\ggo=\kg\oplus\pg,\ip)$ with $B(\kg,\pg)=0$, the Ricci operator satisfies
\begin{equation}\label{alg2}
\Ricci=cI+\unm\left(D_{\pg}+D_{\pg}^t\right), \qquad \mbox{for some}\quad c\in\RR, \quad D=\left[\begin{smallmatrix} 0&0\\ 0&D_\pg\\\end{smallmatrix}\right]\in\Der(\ggo).
\end{equation}
Conversely, if \eqref{alg2} holds for some reductive decomposition and $G/K$ is simply connected, then $(G/K,g)$ is a semi-algebraic soliton.
\end{proposition}

In the case when the derivation $D$ in formula \eqref{alg2} is symmetric, and hence $\Ricci=cI+D_{\pg}$, the homogeneous space $(G/K,g)$ is called an {\it algebraic soliton}.  Actually, all known examples of nontrivial homogeneous Ricci solitons are isometric to a left-invariant algebraic soliton on a simply connected solvable Lie group, the so called {\it solvsolitons}.

The strong presence of the algebraic side of homogeneous manifolds in regard to Ricci soliton theory becomes evident in the following result.

\begin{theorem}\cite[Proposition 2.2]{Jbl}\label{semi}
Any homogeneous Ricci soliton $(M,g)$ is a semi-algebraic soliton with respect to its full
isometry group $G=\Iso(M,g)$.
\end{theorem}

For a recent account of homogeneous Ricci solitons, we refer the reader to \cite[Section 3]{homRS}.

\section{Structure of semi-algebraic solitons}\label{eins}

In this section, we study the algebraic structure of metric reductive decompositions of semi-algebraic solitons.  It will be shown that quite
restrictive Lie theoretic conditions must hold.  The main structural result is Theorem \ref{main}.  Our main tools will be the technical results given in Section \ref{pre}, most of them obtained by using the stratification from
geometric invariant theory described in Appendix \ref{git}.

Those readers only interested in Einstein homogeneous manifolds may just set $D=0$ everywhere, although it is worth mentioning that the effort required to understand the proof of Theorem \ref{main} for either Einstein metrics or semi-algebraic solitons is practically the same.

Let $(G/K,g)$ be a homogeneous space, and consider the metric reductive decomposition $\gkp$ such that $B(\kg,\pg)=0$.  We deduce from Proposition \ref{semi-car}, (\ref{ricci}) and (\ref{Rmm}) that $(G/K,g)$ is a semi-algebraic soliton, say with $\Ricci=cI+S(D_{\pg})$ (see \eqref{alg2}), if and only if
\begin{equation}\label{einstein}
\tr{\left(cI+\unm B_{\pg}+F)\right)E}= \unc\la
\pi(E)[\cdot,\cdot]_{\pg},[\cdot,\cdot]_{\pg}\ra, \qquad\forall E\in\End(\pg),
\end{equation}
where
$$
F:=S(\ad_{\pg}{H}+D_{\pg}).
$$
It follows from  Lemma \ref{Dpp} that the derivation $D\in\Der(\ggo)$ has the following form relative to the decomposition $\ggo=\kg\oplus\hg\oplus\ngo$:
\begin{equation}\label{formD}
D=\left[\begin{smallmatrix}
0&0&0\\
0&D_{\hg}&0\\
0&D_{\hg\ngo}&D_{\ngo}
\end{smallmatrix}\right], \qquad D_{\ngo}\in\Der(\ngo), \qquad \tr{D_\hg} = 0.
\end{equation}

We are now in a position to prove the first structural result, which starts to show that the structure of $\gkp$ is far from arbitrary under the semi-algebraic soliton condition.

\begin{proposition}\label{leo}
Let $(G/K,g)$ be a homogeneous space, and consider the metric reductive decomposition $\gkp$ such that $B(\kg,\pg)=0$.  Assume that $(G/K,g)$ is an expanding semi-algebraic soliton, say with $\Ricci=cI+S(D_{\pg})$, $c<0$, and consider the decompositions $\pg=\hg\oplus\ngo$ and $\lb_{\pg}=\lambda_0+...+\mu$ defined as in {\rm (\ref{decp})} and {\rm (\ref{dec})}, respectively.  Then,
$$
[\hg,\hg]_{\pg}\subset\hg,
$$
or equivalently, $\lambda_1=0$.  Moreover, if $\mu\ne 0$ and $\beta_{\mu}=\beta$ for the $\beta\in\tg^+$ such that $\mu\in\sca_{\beta}$ (see {\rm Theorem \ref{strata}}), then,

\begin{itemize}
\item[(i)] $\beta+||\beta||^2I\in\Der(\ngo)$.
\item[ ]
\item[(ii)] $[\beta,\ad{\hg}|_{\ngo}]=0$.
\item[ ]
\item[(iii)] $S(\ad_{\pg}{H}+D_{\pg})=tE_{\beta}$ for $t=\tfrac{|| H||^2+\tr{D_{\ngo}}}{-1+||\beta||^2\dim{\ngo}}$ (see {\rm (\ref{ebeta})}), or equivalently, the following conditions hold:
    \begin{itemize}
     \item[(a)] $S(\ad_{\lambda_0}{H}+D_{\hg})=0$.

     \item[(b)] $\ad_{\lambda_1}{H}=D_{\hg \ngo} = 0$.

     \item[(c)] $S(\ad_{\pg}{H}|_{\ngo}+D_{\ngo})=t(\beta+||\beta||^2I)$, for some $t\geq 0$.
     \end{itemize}
\end{itemize}

For $\mu=0$ one has that $S(\ad_{\pg}{H}+D_{\pg})=
\left[\begin{smallmatrix}
0&0\\
0&tI
\end{smallmatrix}\right]$, where $t=\tfrac{||H||^2+\tr{D_{\ngo}}}{\dim{\ngo}}$.
\end{proposition}

\begin{proof}
In order to apply the results in Appendix \ref{git}, we identify $\ngo$ with $\RR^n$ via an orthonormal basis $\{ e_1,...,e_n\}$ of $\ngo$.  In this way, $\mu$ can be viewed as an element of $\nca\subset V$. If $\mu\ne 0$ then $\mu\in\sca_{\beta}$ for some $\beta\in\bca\subset\tg^+$  and there exists $h\in\Or(n)$ such that $\tilde{\mu}:=h\cdot\mu$ satisfies $\beta_{\tilde{\mu}}=\beta$; consequently, many extra nice properties of $\tilde{\mu}$ related to $\beta$ hold. Let $\tilde{h}:\ggo\longrightarrow\ggo$ be the map defined by $\tilde{h}|_{\kg\oplus\hg}=I$, $\tilde{h}|_{\ngo}=h$, and let $\tilde{\ggo}$ denote the Lie algebra with underlying vector space $\ggo$ and Lie bracket
$$
\tilde{h}\cdot[\cdot,\cdot]=\tilde{h}[\tilde{h}^{-1}\cdot,\tilde{h}^{-1}\cdot].
$$
It is easy to check that $(\tilde{\ggo}=\kg\oplus\pg,\ip)$ is also a metric reductive decomposition.  Since the condition $\lambda_1=0$ holds for  $(\tilde{\ggo}=\kg\oplus\pg,\ip)$ if and only if it does so for $\gkp$, we can assume
from now on that $\beta_{\mu}=\beta$, and thus we can use in the proof all the properties stated in Theorem \ref{strata} and Lemma \ref{pie}.

\begin{lemma}
Under the same hypotheses of the proposition, we have that
\begin{equation}\label{c}
c\tr{F}+\tr{F^2}=0,
\end{equation}
where $F=S(\ad_{\pg}{H}+D_{\pg})$.
\end{lemma}

\begin{remark}
This lemma actually holds without assuming that $c<0$.
\end{remark}

\begin{proof}
By letting $E=\ad_{\pg}{H}+D_{\pg}$ in (\ref{einstein}) we get
\begin{equation}\label{c1}
c\tr{F}+\unm\tr{B_{\pg}(\ad_{\pg}{H}+D_\pg)}+\tr{F^2}=
\unc\la\pi(\ad_{\pg}{H}+D_{\pg})\lb_{\pg},\lb_{\pg}\ra.
\end{equation}
According to \eqref{kH0}, Lemma \ref{Dpp} and \eqref{formD}, we have that
\[
\tilde{E}:=\left[ \begin{smallmatrix}  0&0\\ 0&\ad_{\pg}{H}+D_{\pg} \end{smallmatrix}\right] = \ad{H} + D \in \Der(\ggo).
\]
Thus $\ad_{\pg}{H}+D_{\pg}$ is a derivation of the algebra $(\pg,\lb_\pg)$, and so the right hand side of \eqref{c1} vanishes and from Remark \ref{Bort} we obtain that $\tr{B_{\pg}(\ad_{\pg}{H}+D_\pg)} = 0$. This concludes the proof of the lemma.
\end{proof}

Let us first assume that $\mu=0$, and apply (\ref{einstein}) to $E\in\End(\ggo)$ defined by $E|_{\hg}=0$, $E|_{\ngo}=I$.  Recall that $\tr{F|_{\ngo}}=\tr{F}$ thanks to Lemma \ref{Dpp}.  We therefore obtain from (\ref{einstein}) that
$$
\unc||\lambda_1||^2= c n+\tr{F}.
$$
If $n=0$ the claim is trivial. Otherwise, since $c<0$ we must have that $\tr{F}>0$, and so from (\ref{c}) we get
\begin{align*}
\unc||\lambda_1||^2 = & cn+\tr{F} = -\tfrac{\tr{F^2}}{\tr{F}}n+\tr{F} = \tfrac{\tr{F^2}}{\tr{F}} \left(\tfrac{(\tr{F})^2}{\tr{F^2}}-n\right) \\
\leq & \tfrac{\tr{F^2}}{\tr{F}} \left(\tfrac{(\tr{F|_{\ngo}})^2}{\tr{(F|_{\ngo})^2}} -n\right) \leq \tfrac{\tr{F^2}}{\tr{F}}
\left(\tfrac{(f_1+...+f_n)^2}{f_1^2+...+f_n^2}-n\right) \leq 0,
\end{align*}
where $f_1,\dots,f_n$ are the eigenvalues of $F|_{\ngo}$.  Thus $\lambda_1=0$, and we get in addition that $F|_\hg=0$ (since $\tr{F^2}=\tr{(F|_\ngo)^2}$) and also that $F|_{\ngo}=tI$ for some
$t\in\RR$, which must satisfy $t=\tfrac{\tr{F|_{\ngo}}}{n}=\tfrac{||H||^2+\tr{D_{\ngo}}}{n}$.

We now consider the case $\mu\ne 0$.  Recall that we can assume that $\mu$ satisfies $\beta_{\mu}=\beta$, and thus the right hand side of (\ref{einstein}) applied to $E_{\beta}$ (see (\ref{ebeta})) is $\geq 0$ by Lemma \ref{pie}. We therefore obtain from (\ref{einstein}) that
\begin{equation}\label{prueba1}
c\tr{E_{\beta}}+\tr{FE_{\beta}}\geq 0.
\end{equation}
In particular, we have that $F\ne 0$ since $c<0$ and $\tr{E_{\beta}}>0$ (see (\ref{betapos})), which implies that $\tr{F}>0$ and $c=-\tfrac{\tr{F^2}}{\tr{F}}$ by (\ref{c}). Recall that $\tr{\beta}=-1$ and hence
\begin{align}
\tr{E_{\beta}^2} =&\tr(\beta^2+||\beta||^4I+2||\beta||^2\beta)=||\beta||^2(1+n||\beta||^2-2) \label{prueba3} \\
=& ||\beta||^2(-1+n||\beta||^2)=||\beta||^2\tr{E_{\beta}}. \notag
\end{align}
On the other hand, we have by (\ref{betaort}) that
\begin{equation}\label{prueba4}
\tr{FE_{\beta}}=\tr{F|_{\ngo}(\beta+||\beta||^2)}=||\beta||^2\tr{F}.
\end{equation}
We now use (\ref{c}), (\ref{prueba1}), (\ref{prueba3}) and
(\ref{prueba4}) to obtain by a straightforward manipulation that
$$
\tr{F^2}\tr{E_{\beta}^2}\leq (\tr{FE_{\beta}})^2,
$$
a `backwards' Cauchy-Schwartz inequality.  This implies that $F=tE_{\beta}$ for some $t>0$, and so part (iii) follows (recall that conditions (a)-(c) hold by
(\ref{adHB}) and (\ref{formD})).  But we also get that equality holds in (\ref{prueba1}), which turns all the inequalities obtained in the proof of Lemma \ref{pie} into equalities.  In particular, $\lambda_1=0$ follows from (\ref{l2l3e1}), and part (i) does so from (\ref{mu}) and (\ref{delta}).

Finally, by using (\ref{l4}) and (\ref{adbeta}) one obtains part (ii), concluding the proof of the proposition.
\end{proof}

According to Proposition \ref{leo}, if $\gkp$ is a metric reductive decomposition of an expanding semi-algebraic soliton, then the matrices
of the adjoint maps and $D$ have the following simpler forms:
\begin{equation}\label{adH}
\begin{array}{cc}
\ad{H}=\left[\begin{smallmatrix}
0&0&0\\
0&\ad_{\lambda_0}{H}&0\\
0&0&\ad_{\eta}{H}\end{smallmatrix}\right], &

D=\left[\begin{smallmatrix}
0&0&0\\
0&D_{\hg}&0\\
0&0&D_\ngo\end{smallmatrix}\right],
\end{array}
\end{equation}
and for all $Y\in\hg$,
\begin{equation}\label{adYA}
\ad{Y}=\left[\begin{smallmatrix} 0&\ad_{\lambda_2}{Y}&0\\
\ad_{\nu_1}{Y}&\ad_{\lambda_0}{Y}&0\\
0&0&\ad_{\eta}{Y}\end{smallmatrix}\right].
\end{equation}

After the structural results obtained in Proposition \ref{leo}, we can now see what is the formula for $\mm $, the most complicated term of
$\Ricci$.

\begin{lemma}\label{calculoR}
Let $(G/K,g_{\ip})$ be a homogeneous space with any metric reductive decomposition $\gkp$, and consider the decompositions $\pg=\hg\oplus\ngo$ and $\lb_{\pg}=\lambda_0+...+\mu$ defined in {\rm (\ref{decp})} and {\rm (\ref{dec})}, respectively.  Suppose that $[\hg,\hg]\subseteq \kg \oplus \hg$ (i.e. $\lambda_1=0$). Then the symmetric operator $\mm $ defined in {\rm (\ref{R})} is given for all $Y\in\hg$, $X\in\ngo$ by
$$
\begin{array}{l}
\la \mm Y,Y\ra=\la \mm_{\lambda_0}Y,Y\ra-\unm\tr{\ad_{\eta}{Y}(\ad_\eta{Y})^t}, \\ \\
\la \mm X,X\ra=\la \mm_{\mu}X,X\ra+\unm\sum\la[\ad_\eta{Y_i},(\ad_\eta{Y_i})^t]X,X\ra,  \\ \\
\la \mm Y,X\ra=-\unm\tr{\ad_\eta{Y}(\ad_{\mu}{X})^t},
\end{array}
$$
where $\mm_{\lambda_0}$ and $\mm_{\mu}$ are defined as in {\rm (\ref{R})} by replacing $\lb_{\pg}$ with the brackets $\lambda_0:\hg\times\hg\longrightarrow\hg$ and $\mu:\ngo\times\ngo\longrightarrow\ngo$, respectively, and $\{ Y_i\}$ is any orthonormal basis of $\hg$.
\end{lemma}

\begin{remark}\label{remcalculoR}
The lemma actually holds for any direct sum decomposition $\pg = \hg \oplus \ngo$ in subspaces such that $[\hg,\hg]_\pg\subset\hg$ and $[\pg,\ngo]_\pg\subset\ngo$.
\end{remark}

\begin{proof}
Let $\{ X_i\}$ be an orthonormal basis of $\ngo$.  Since $\lambda_1=0$, the computation of $\mm $ for $Y\in\hg$, $X\in\ngo$ can be done as follows:
\begin{align*}
\la \mm Y,Y\ra =& -\unm\sum\la\lambda_0(Y,Y_i),Y_j\ra^2 -\unm\sum\la\eta(Y,X_i),X_j\ra^2 \\ &+ \unc\sum\la\lambda_0(Y_i,Y_j),Y\ra^2   \\
=& \la \mm_{\lambda_0}Y,Y\ra -\unm\sum\la\ad_{\eta}{Y}(X_i),\ad_{\eta}{Y}(X_i)\ra \\
=& \la \mm_{\lambda_0}Y,Y\ra  -\unm\sum\la(\ad_{\eta})^t{Y}\ad_{\eta}{Y}(X_i),X_i\ra \\
=& \la \mm_{\lambda_0}Y,Y\ra -\unm\tr{(\ad_{\eta}{Y})^t\ad_{\eta}{Y}},
\end{align*}

\begin{align*}
\la \mm X,X\ra =& -\unm\sum\la[X,Y_i],X_j\ra^2 -\unm\sum\la\mu(X,X_i),X_j\ra^2 \\
&+\unm\sum\la\eta(Y_i,X_j),X\ra^2 + \unc\sum\la\mu(X_i,X_j),X\ra^2 \\
=& \la \mm_{\mu}X,X\ra - \unm\sum\la\eta(Y_i,X),X_j\ra^2 +\unm\sum\la\eta(Y_i,X_j),X\ra^2 \\
=& \la \mm_{\mu}X,X\ra - \unm\sum\la\ad_{\eta}{Y_i}(X),\ad_{\eta}{Y_i}(X)\ra \\
&+ \unm\sum\la(\ad_{\eta}{Y_i})^t(X),(\ad_{\eta}{Y_i})^t(X)\ra \\
=& \la \mm_{\mu}X,X\ra + \unm\sum\la[\ad_{\eta}{Y_i},(\ad_{\eta}{Y_i})^t]X,X\ra,
\end{align*}

\begin{align*}
\la \mm Y,X\ra &= -\unm\sum\la\eta(Y,X_i),X_j\ra\la\mu(X,X_i),X_j\ra \\
&= -\unm\tr{\ad_{\eta}{Y}(\ad_{\mu}{X})^t},
\end{align*}
which concludes the proof of the lemma.
\end{proof}

We are finally prepared to prove the main structural result of this paper.

\begin{theorem}\label{main}
Let $(G/K,g)$ be a homogeneous space, and consider the metric reductive decomposition $\gkp$ such that $B(\kg,\pg)=0$.  Also consider the orthogonal decomposition $\pg=\hg\oplus\ngo$, where $\ngo$ is the nilradical of $\ggo$.  If $(G/K,g)$ is an expanding semi-algebraic soliton with cosmological constant $c<0$, then the following conditions hold:
\begin{itemize}
\item[(i)] $[\hg,\hg]\subset\kg\oplus\hg$.  In particular, $\ug=\kg\oplus\hg$ is a reductive Lie subalgebra of $\ggo$ and $\ggo=\ug\ltimes\ngo$ (semidirect product).
\item[ ]
\item[(ii)] $\Ricci_{\ug}=cI+C_{\hg}$, where $\Ricci_{\ug}$ is the Ricci operator of the metric reductive decomposition $(\ug=\kg\oplus\hg,\ip|_{\ug\times\ug})$ and $C_{\hg}$ is the symmetric map defined by
    $$
    \la C_{\hg}Y,Y\ra = \tr{S(\ad{Y}|_{\ngo})^2}, \qquad\forall Y\in\hg.
    $$
\item[(iii)] $\Ricci_{\ngo}=cI+D_1$, for some $D_1 \in \Der(\ngo)$, where $\Ricci_{\ngo}$ denotes the Ricci operator of the metric nilpotent Lie algebra $(\ngo,\ip|_{\ngo\times\ngo})$ (i.e. $(\ngo,\ip|_{\ngo\times\ngo})$ is a nilsoliton).
\item[ ]
\item[(iv)] $\sum [\ad{Y_i}|_{\ngo},(\ad{Y_i}|_{\ngo})^t]=0$, where $\{ Y_i\}$ is any orthonormal basis of $\hg$ (in particular, $(\ad{Y}|_{\ngo})^t\in\Der(\ngo)$ for all $Y\in\hg$).

\item[ ]
\item[(v)] The Ricci operator of $(G/K,g)$ is given by $\Ricci=cI+S(D_{\pg})$, where
$$
D:=-\ad{H}+ \left[\begin{smallmatrix} 0&&\\ &0&\\ &&D_1\\\end{smallmatrix}\right]\in\Der(\ggo).
$$
\end{itemize}
Conversely, if conditions {\rm (i)-(iv)} hold for some metric reductive decomposition (not necessarily with $B(\kg,\pg)=0$) and $G/K$ is simply connected, then $(G/K,g)$ is a semi-algebraic soliton with cosmological constant $c$ and derivation $D$ as above.
\end{theorem}

\begin{remark}
In particular, these structural results apply to any Einstein homogeneous space of negative scalar curvature, as such spaces are all expanding semi-algebraic solitons with respect to any reductive decomposition.
\end{remark}

\begin{remark}
In part (ii), $\Ricci_{\ug}$ is defined as in (\ref{ricci}) and it is actually the Ricci operator of the homogeneous space $U/K_0$ endowed with the $U$-invariant metric determined by $\ip|_{\hg\times\hg}$, where $U$ is the connected Lie subgroup of $G$ with Lie algebra $\ug$ and $K_0$ is the connected component of the identity of $K$.  We note that $U/K_0$ is almost-effective if and only if the kernel of the map $\kg\longrightarrow\End(\hg)$, $Z\mapsto\ad{Z}|_\hg$, does vanish.
\end{remark}

\begin{proof}
We consider the metric reductive decomposition $(\wt{\ggo}=\kg\oplus\pg,\ip)$ defined at the beginning of the proof of Proposition \ref{leo}, for which the corresponding $\wt{\mu}=h\cdot\mu$ ($h\in\Or(n)$) satisfies $\beta_{\wt{\mu}}=\beta$. The Ricci operator of $(\wt{\ggo}=\kg\oplus\pg,\ip)$ is given by $\wt{\Ricci}=h \Ricci h^{-1}$.  Moreover,  it is easy to see that
$$
\wt{\Ricci}_{\ug}=\Ricci_{\ug}, \quad \wt{\ad}{Y}|_{\ngo}=h\ad{Y}|_{\ngo}h^{-1}, \quad \wt{C}_{\hg}=C_{\hg}, \quad \wt{\Ricci}_{\ngo}=h\Ricci_{\ngo}h^{-1},
$$
for all $Y\in\hg$, from which follows that conditions (i)-(iv) hold for $(\ggo=\kg\oplus\pg,\ip)$ if and only if they do so for $(\wt{\ggo}=\kg\oplus\pg,\ip)$.  This allows us to assume from now on that condition $\beta_{\mu}=\beta$ holds for our $\mu$ if nonzero and so we can use all the properties stated in Theorem \ref{strata} and Lemma \ref{pie}.

We first suppose that $(G/K,g)$ is a semi-algebraic soliton as stated in the theorem, say with $\Ricci=cI+S(D_{\pg})$, $c<0$.  Recall that some restrictions on the form of $D$ have already been proved (see (\ref{adH})).  The notation introduced in Section \ref{pre} will be constantly used along the proof of the theorem. Part (i) follows from Proposition \ref{leo}. In order to prove parts (iii) and (iv) we need the following result.

\begin{lemma}\label{adY}
For any $Y\in\hg$, $(\ad{Y}|_{\ngo})^t$ is a derivation of $\ngo$.
\end{lemma}

\begin{proof}
For $\mu=0$ (i.e. $\ngo$ abelian), the assertion is trivially true.  If $\mu\ne 0$ and $\mu\in\sca_{\beta}$ then we know from Proposition \ref{leo}, (iii), (c) that $F|_{\ngo}=t(\beta+||\beta||^2I)$, where $F=S(\ad_{\pg}{H}+D_{\pg})$, and since $F|_\ngo$ is a derivation of $\ngo$ by Proposition \ref{leo}, (i), we get from by (\ref{betaort}) that
$$
\tr{(F|_{\ngo})^2} = t||\beta||^2\tr{F|_{\ngo}},
$$
But it follows from Proposition \ref{leo}, (iii) that $\tr{F|_\ngo} = \tr{F}$ and $\tr{(F|_\ngo)^2} = \tr{F^2}$, so we can use \eqref{c} to conclude that $t=-\tfrac{c}{||\beta||^2}$ and thus
\begin{equation}\label{p3}
F|_{\ngo}=-\tfrac{c}{||\beta||^2}\beta-cI.
\end{equation}
Recall that $F\ne 0$ by (\ref{prueba1}), and consequently $\tr{F|_{\ngo}}=\tr{F}>0$ by (\ref{einstein}).

On the other hand, since $\Ricci|_{\ngo}=cI+S(D_{\ngo})$, it follows from Lemma \ref{calculoR} (which we can use thanks to part (i)) that
$$
\mm_{\mu} +\unm\sum\left[\ad_{\eta}{Y_i},(\ad_{\eta}{Y_i})^t\right]- F|_{\ngo}=cI,
$$
and hence (\ref{p3}) implies the following equality on $\ngo$:
\begin{equation}\label{p4}
\mm_{\mu}+\unm\sum\left[\ad_{\eta}{Y_i},(\ad_{\eta}{Y_i})^t\right]
+\tfrac{c}{||\beta||^2}\beta =0.
\end{equation}
By taking trace in the above equality and using that $\tr{\mm_{\mu}}=-\unc||\mu||^2$ and $\tr{\beta}=-1$, we obtain that
\begin{equation}\label{p5}
c=-\unc||\beta||^2||\mu||^2.
\end{equation}
It also follows from (\ref{p4}) that
$$
0= \tr{\mm_{\mu}^2}+
\unm\sum\tr{\mm_{\mu}\left[\ad_{\eta}{Y_i},(\ad_{\eta}{Y_i})^t\right]}
+\tfrac{c}{||\beta||^2}\tr{\mm_{\mu}\beta}.
$$
Analogously to (\ref{Rmm}), we have that
$$
\tr{\mm_{\mu}E}=\unc\la \pi(E)\mu,\mu\ra, \qquad\forall E\in\End(\ngo),
$$
or equivalently, $m(\mu)=\tfrac{4}{||\mu||^2}\mm_{\mu}$ for all $\mu\in V$ (see (\ref{mmv})).  Thus, for all $i$,
\begin{align}
\tr{\mm_{\mu}[\ad_{\eta}{Y_i},(\ad_{\eta}{Y_i})^t]} =& \unc\la\pi([\ad_\eta{Y_i}|_{\ngo},(\ad_\eta{Y_i})^t])\mu,\mu\ra \notag \\
=& \unc\la\pi(\ad_\eta{Y_i})\pi((\ad_\eta{Y_i})^t)\mu,\mu\ra \label{adYtder} \\
=& \unc\la\pi((\ad_\eta{Y_i})^t)\mu,\pi((\ad_\eta{Y_i})^t)\mu\ra \notag \\
=& \unc||\pi((\ad_\eta{Y_i})^t)\mu||^2. \notag
\end{align}
This and (\ref{p5}) give
\begin{align*}
0=& \tr{\mm_{\mu}^2}+ \tfrac{1}{8}\sum||\pi((\ad_{\eta}{Y_i})^t)\mu||^2
-\tfrac{||\mu||^2}{4}\la \mm_{\mu},\beta\ra \\
=& \tfrac{1}{8}\sum||\pi((\ad_{\eta}{Y_i})^t)\mu||^2
+\tfrac{||\mu||^4}{16}\left(\tfrac{16}{||\mu||^4}\tr{\mm_{\mu}^2}
-\left\la\tfrac{4}{||\mu||^2}\mm_{\mu},\beta\right\ra\right)\\
=& \tfrac{1}{8}\sum||\pi((\ad_{\eta}{Y_i})^t)\mu||^2
+\tfrac{||\mu||^4}{16}\left(|| m(\mu)||^2 -\left\la m(\mu),\beta\right\ra\right).
\end{align*}
But since $\mu$ satisfies $\beta_{\mu}=\beta$ we obtain from (\ref{bmu}) that
$$
\la m(\mu),\beta\ra\leq ||m(\mu)||\;||\beta||\leq || m(\mu)||^2,
$$
which implies that $\tfrac{1}{8}\sum||\pi((\ad_{\eta}{Y_i})^t)\mu||^2=0$.  Thus $(\ad_{\eta}{Y_i})^t\in\Der(\ngo)$ for all $i$ and the lemma follows.
\end{proof}

\begin{remark}\label{Dnorm}
By applying \eqref{adYtder} to any metric Lie algebra, one obtains that the transpose of any normal derivation is always a derivation as well.
\end{remark}

If $\mu=0$, then parts (iii) and (iv) follow directly from the fact that $F|_{\ngo}=tI$ (see Proposition \ref{leo}) and the formula for $\Ricci|_{\ngo}$ and $\mm |_{\ngo}$ (see Lemma \ref{calculoR}).  We also obtain that $t=-c$.

For $\mu\ne 0$, we obtain from (\ref{p4}) and Lemma \ref{adY} that
$$
\mm_{\mu}+\tfrac{c}{||\beta||^2}\beta=0, \qquad
\sum\left[\ad_{\eta}{Y_i},(\ad_{\eta}{Y_i})^t\right]=0,
$$
as they are orthogonal maps by Remark \ref{Rort} and (\ref{betaort}). This implies part (iv), and part (iii) follows from the first equation above, since by (\ref{p3}),
$$
\mm_{\mu}=-\tfrac{c}{||\beta||^2}\beta =cI+F|_{\ngo}.
$$
Regarding (ii), we know from Lemma \ref{calculoR} that
\begin{equation}\label{p6}
\la \mm_{\lambda_0}Y,Y\ra = \la \mm Y,Y\ra + \unm\tr{\ad_{\eta}{Y}(\ad_\eta{Y})^t}.
\end{equation}
On the other hand, since $\ug$ is unimodular as it is reductive, we have
\begin{equation}\label{p7}
\la \Ricci_\ug Y, Y \ra = \la \mm_{\lambda_0}Y,Y\ra - \unm \tr{(\ad{Y}|_\ug)^2}, \quad \forall Y\in \hg.
\end{equation}
So, using that $F|_\hg = 0$, $\Ricci|_\hg = c I$ and $\tr{(\ad{Y})^2} = \tr{(\ad{Y}|_\ug)^2} + \tr{(\ad_\eta {Y})^2}$, together with \eqref{p6} and \eqref{p7}, yields the desired formula for $\Ricci_\ug$.  This concludes the proof of the first part of the proposition.

\vs

Conversely, let us assume that conditions (i)-(iv) hold.  We first recall that (iv) implies that $(\ad_{\eta}{Y})^t$ is a derivation of $\ngo$ for any $Y\in\hg$, by using \eqref{adYtder}. Also, since we have by (i) that $\lambda_1 = 0$ we may use Lemma \ref{calculoR} and compute $\Ricci$ as follows.

In the first place, (ii) and \eqref{p7} tell us that
\begin{align}
\la(\mm - \unm B_\pg) Y,Y\ra =& \la \mm_{\lambda_0}Y,Y\ra - \unm\tr{\ad_{\eta}{Y}(\ad_\eta{Y})^t} - \unm\tr{(\ad{Y})^2} \notag  \\
=& \la \Ricci_\ug Y, Y \ra + \unm \tr{(\ad{Y}|_\ug)^2} - \unm\tr{(\ad{Y})^2} \notag \\ & - \unm\tr{\ad_{\eta}{Y}(\ad_\eta{Y})^t} \notag \\
=& \la \Ricci_\ug Y, Y \ra  - \unm\tr{(\ad_\eta {Y})^2} - \unm\tr{\ad_{\eta}{Y}(\ad_\eta{Y})^t} \label{ric1} \\
=& c \| Y \|^2, \notag
\end{align}
for all $Y\in \hg$, by part (ii). Moreover, from \eqref{adHB} one gets
\begin{equation}\label{ric2}
\la(\mm - \unm B_\pg) Y,X\ra = -\unm\tr{\ad_\eta{Y}(\ad_{\mu}{X})^t} = 0, \quad \forall Y\in \hg, X\in \ngo,
\end{equation}
since $(\ad_{\eta}{Y})^t$ is also a derivation of $\ngo$ and so if $\ngo=\ngo_1\oplus...\oplus\ngo_r$ is the orthogonal decomposition with $[\ngo,\ngo]=\ngo_2\oplus...\oplus\ngo_r$, $[\ngo,[\ngo,\ngo]]=\ngo_3\oplus...\oplus\ngo_r$, and so on, then $\ad_{\eta}{Y}$ leave the subspaces $\ngo_i$ invariant and $\ad_{\mu}{X}(\ngo_i)\subset\ngo_{i+1}\oplus...\oplus\ngo_r$ for all $i$. And also from \eqref{adHB} and (iv) we obtain
\begin{equation}\label{ric3}
\la(\mm - \unm B_\pg) X,X\ra = \la \mm_\mu X,X\ra = \la \Ricci_\ngo X, X \ra.
\end{equation}
By putting \eqref{ric1}, \eqref{ric2} and \eqref{ric3} together and by using (iii) we conclude that
\[
\Ricci = cI + S(D_\pg), \qquad \mbox{for} \quad D = -\ad{H} + \tilde{D_1},
\]
where $\tilde{D_1}$ is defined by $\tilde{D_1}|_{\kg \oplus \hg} = 0$, $\tilde{D_1}|_\ngo = D_1$.  It only remains to show that $\tilde{D_1} \in \Der(\ggo)$, as we can then apply Proposition \ref{semi-car} to obtain that $(G/K,g)$ is a semi-algebraic soliton.

To do that, recall that $D_1\in \Der(\ngo)$ and so it is easy to see that it suffices to prove that
\begin{equation}\label{p8}
[\ad{(Y+Z)}|_\ngo, D_1] = 0,
\end{equation}
for all $Y\in \hg$, $Z\in \kg$. But $D_1$ commutes with every derivation $E$ of $\ngo$ whose transpose is also a derivation since $M_\mu=cI+D_1$ (see Remark \ref{Mcomm}), and therefore
\eqref{p8} follows from the fact that $(\ad(Y)|_\ngo)^t\in \Der(\ngo)$ and $(\ad{Z}|_\ngo)^t = -\ad Z|_{\ngo}$, for all $Y\in \hg$, $Z\in \kg$.  This concludes the proof of the theorem.
\end{proof}

\begin{corollary}\label{solv-red}
Let $(G/K,g)$ be an expanding semi-algebraic soliton as in Theorem \ref{main}.

\begin{itemize}
\item[(i)] If $[\hg,\hg]=0$, then $(G/K,g)$ is isometric to a solvsoliton.  This in particular holds when $\ggo$ is solvable.

\item[(ii)] If $\ggo$ is semisimple, then $(G/K,g)$ is Einstein with $\ricci(g)=cg$.
\end{itemize}
\end{corollary}

\begin{remark}
Part (ii) has already been proved in \cite[Theorem 1.6]{Jbl}.
\end{remark}

\begin{proof}
We first prove part (i).  It is easy to see that $(G/K,g)$ is isometric to the left-invariant metric defined by $\ip$ on the connected solvable Lie subgroup $S$ of $G$ with Lie algebra $\hg\oplus\ngo$, which is easily seen to be a solvsoliton by using parts (ii)-(iv) of Theorem \ref{main} and applying \cite[Proposition 4.3]{solvsolitons}.  Recall that $G$ can be assumed to be simply connected without losing almost-effectiveness, and so $S$ is simple connected as $G=K\times S$ as differentiable manifolds.  We note that if $\ggo$ is solvable then actually the whole Lie subalgebra $\ug$ must be abelian.

Part (ii) follows from the fact that $\ngo=0$ and so $C_\hg=0$ (see Theorem \ref{main}, (ii)).
\end{proof}

Along the proof of Theorem \ref{main}, the following extra structural properties
have been obtained.

\begin{proposition}\label{extras}
Under the same hypothesis of Theorem \ref{main}, if say $\Ricci=cI+S(D_\pg)$, and $F = S(\ad_\pg H + D_\pg)$, then $D_1=S(\ad{H}|_{\ngo}+D|_{\ngo})$ and
$$
[\ad{\ug}|_\ngo, D_1]=0, \qquad [\ad \ug |_\pg, F] = 0.
$$
Also, the operator $M$ leaves $\ngo$ invariant, and $M|_\ngo = M_\mu$.  Assume now that $\ngo$ is not abelian, $\mu:=\lb|_{\ngo\times\ngo}\in\sca_{\beta}$
and $\beta_{\mu}=\beta$, and define $E_{\beta}\in\End(\ggo)$ by
$$
E_{\beta}:=\left[\begin{smallmatrix} 0&&\\ &0&\\
&&\beta+||\beta||^2I\end{smallmatrix}\right], \qquad\mbox{i.e.}\quad
E_{\beta}|_{\kg\oplus\hg}=0, \quad E_{\beta}|_{\ngo}=\beta+||\beta||^2I.
$$
Then the following conditions hold:
\begin{itemize}
\item $E_{\beta}\in\Der(\ggo)$ (or equivalently, $[\beta,\ad{\ug}|_{\ngo}]=0$
and $\beta+||\beta||^2I\in\Der(\ngo)$).

\item $S(\ad{H}+D)=-\tfrac{c}{||\mu||^2}E_{\beta}$.

\item $c=-\unc ||\mu||^2||\beta||^2$ and the moment map satisfies $m(\mu)=\beta$.
\end{itemize}
\end{proposition}

In \cite{homRS}, algebraic solitons have been geometrically characterized among homogeneous Ricci solitons as those for which the Ricci flow solution is simultaneously diagonalizable with respect to a fixed orthonormal basis of some tangent space.  As a first application of Theorem \ref{main}, we now give some structural characterizations of algebraic solitons.

\begin{proposition}\label{equiv-algsol}
Under the same hypothesis of Theorem \ref{main}, assume that $(G/K,g)$ is an expanding semi-algebraic soliton with $\Ricci=cI+S(D_{\pg})$, $c<0$, $D=\left[\begin{smallmatrix} 0&0\\ 0&D_\pg\\\end{smallmatrix}\right]\in\Der(\ggo)$. Then the following conditions are equivalent:
\begin{itemize}
  \item[(i)] $(G/K,g)$ is an algebraic soliton (i.e. $S(D) \in \Der(\ggo)$).
  \item[(ii)] $S(D_\pg) \in \Der(\lb_\pg)$.
  \item[(iii)] $S(\ad_\pg H) \in \Der(\lb_\pg)$ (or equivalently, $S(\ad{H})\in\Der(\ggo)$).
  \item[(iv)] $\ad_\pg H$ is normal (or equivalently, $\ad{H}$ is normal).
  \item[(v)] $S(D|_\hg) = 0$.
  \item[(vi)] $S(\ad{H}|_\hg) = 0$.
  \item[(vii)] $\Ricci |_\hg = cI$.
\end{itemize}
\end{proposition}

\begin{proof}
If one assumes part (i), then (ii) holds by $S(D)\pg \subset \pg$.  Conversely, part (ii) together with the fact that $[\ad Z|_\pg, S(D_\pg)] = 0$ for all $Z\in\kg$ (which follows from \eqref{Z}), imply that $S(D)=\left[\begin{smallmatrix} 0&0\\ 0&S(D_\pg)\\\end{smallmatrix}\right]\in\Der(\ggo)$.

The equivalence between part (iii) and $S(\ad{H})\in\Der(\ggo)$ follows as above by using that $S(\ad{H})\pg\subset\pg$ (see \eqref{adH}) and $[\ad{\kg}|_\pg,S(\ad_\pg{H})]=0$ (see \eqref{kH0}).  It follows from Proposition \ref{extras} that $S(\ad_\pg H + D_\pg) \in \Der(\lb_\pg)$, thus parts (ii) and (iii) are equivalent. To see that (iii) implies (iv), first observe that for all $X\in \pg$,
\[
\la \left(\ad_\pg H\right)^t H,X \ra = \la H, [H,X]_\pg \ra =  \tr \ad[H,X] = \tr [\ad H, \ad X] = 0,
\]
so $\left(\ad_\pg H\right)^t H = 0$.  By using this and the fact that $\left(\ad_\pg H\right)^t$ is a derivation (which is equivalent to part (iii), since $\ad_\pg H \in \Der(\lb_\pg)$), we get
\[
\left(\ad_\pg H\right)^t \left([H,X]_\pg \right) = [H,\left(\ad_\pg H\right)^t X ]_\pg.
\]
and thus $\ad_\pg H $ is normal ($\ad{H}$ is therefore normal by \eqref{adH}).  Conversely, if part (iv) holds, then we can argue as in \eqref{adYtder} to obtain
\[
0 = \tr M [\ad_\pg H, \left( \ad_\pg H\right)^t ] = \tfrac14 \left\|\pi\left( \left( \ad_\pg H\right)^t \right) \lb_\pg \right\|^2.
\]
This implies that $\left(\ad_\pg H\right)^t \in \Der(\lb_\pg)$, and thus part (iii) follows.

From Proposition \ref{extras} we have that $S(\ad H |_\hg+D|_\hg) = 0$, therefore using that $\Ricci |_\hg = cI + S(D|_\hg)$ it easily follows that parts (v), (vi) and (vii) are pairwise equivalent.

Now if (vi) holds, then we use the fact that $M|_\ngo = M_\mu$ (see Proposition \ref{extras}) to obtain
\begin{align*}
\tfrac14 \left\|\pi\left( \left( \ad_\pg H\right)^t \right) \lb_\pg \right\|^2 =& \tr M [\ad_\pg H, \left( \ad_\pg H\right)^t ] \\
=& \tr M_\mu [\ad{H}|_\ngo, \left( \ad{H}|_\ngo\right)^t ] \\
=& \tfrac14 \left\|\pi\left( \left( \ad{H}|_\ngo\right)^t \right) \mu\right\|^2 = 0,
\end{align*}
by Theorem \ref{main}, (iv), from which part (iii) follows.

Conversely, assume that $\left( \ad_\pg H \right)^t \in \Der(\lb_\pg)$. This implies that $[M, \ad_\pg H] = 0$ by Remark \ref{Mcomm}. On the other hand, recall that $M-\unm B_\pg = cI + F$, and by Proposition \ref{extras} $cI + F$ also commutes with $\ad_\pg H$. Hence, we must have that
\[
    [B_\pg, \ad_\pg H] = 0.
\]
Now if we decompose $\pg = \hg_1 \oplus \ag \oplus \ngo$ as in the proof of Lemma \ref{Dpp}, then relative to that decomposition the operators $B_\pg$ and $\ad_\pg H$ have the form:
\[
B_\pg = \left[\begin{smallmatrix}
B_2&0&0\\
0&0&0\\
0&0&0
\end{smallmatrix}\right], \qquad
\ad_\pg H = \left[\begin{smallmatrix}
\ad H |_{\hg_1} &0&0\\
0&0&0\\
0&0&\ad H |_\ngo
\end{smallmatrix}\right].
\]
Here, we are using \eqref{formBD} and the fact that $(\ad H)^t\in\Der(\ggo)$. Also, as in the proof of Lemma \ref{Dpp}, we obtain that
\[
    \left(\ad H |_{\hg_1}\right)^t = -B_2 \left(\ad H |_{\hg_1}\right) B_2^{-1}.
\]
But $[B_2,\ad H |_{\hg_1}] = [B_\pg, \ad_\pg H]|_{\hg_1} = 0$, hence $\left(\ad H |_{\hg_1}\right)^t = - \ad H |_{\hg_1}$ and part (vi) follows. This concludes the proof of the proposition.
\end{proof}

\section{Construction procedure for semi-algebraic solitons}\label{cons}

Our aim in this section is to state Theorem \ref{main} in a more transparent way via a
construction procedure.

First consider the following data set satisfying the following conditions:

\begin{itemize}
\item[(d1)] $(\ngo,\ip_{\ngo})$: a metric nilpotent Lie algebra with Ricci operator
$\Ricci_{\ngo}=cI+D_1$, for some $c<0$, $D_1\in\Der(\ngo)$ (i.e. a
nilsoliton $(\ngo,\ip_{\ngo})$; recall that $D_1$ is always positive definite).
\item[ ]

\item[(d2)] $(\ug=\kg\oplus\hg,\ip_{\ug})$: a metric reductive decomposition with $\ug$ a reductive Lie algebra.
\item[ ]

\item[(d3)] $\theta:\ug\longrightarrow\Der(\ngo)$: a homomorphism of Lie algebras such that
\begin{itemize}
\item[ ]
\item[(c1)] $\theta(Z)^t=-\theta(Z)$ for all $Z\in\kg$.
\item[ ]

\item[(c2)] $\sum [\theta(Y_i),\theta(Y_i)^t]=0$ for any orthonormal basis $\{ Y_i\}$ of $\hg$ (it follows as in \eqref{adYtder} that $\theta(Y)^t\in\Der(\ngo)$ for any $Y\in\hg$).
\item[ ]

\item[(c3)] The Ricci operator of $(\ug=\kg\oplus\hg,\ip_{\ug})$ satisfies
$$
\Ricci_{\ug}=cI+C_{\theta}, \qquad \mbox{where} \quad \la C_{\theta}Y,Y\ra=\tr{S(\theta(Y))^2}, \quad\forall Y\in\hg.
$$
\end{itemize}
\end{itemize}

Take now the corresponding semidirect product of Lie algebras,
$$
\ggo=\ug\oplus\ngo,
$$
and the metric reductive decomposition $\gkp$, where
$\pg:=\hg\oplus\ngo$ and the inner product $\ip$ is given by
$$
\ip|_{\ug\times\ug}=\ip_{\ug}, \qquad
\ip|_{\ngo\times\ngo}=\ip_{\ngo},
$$
and such that
$$
    \ggo= \rlap{$\overbrace{\phantom{\kg\oplus\hg}}^\ug$} \kg \oplus \underbrace{\hg\oplus\ngo}_\pg
$$
is an orthogonal decomposition.  Define $D:\ggo\longrightarrow\ggo$ by
$$
D:=-\left[\begin{smallmatrix} \ad_\ug{H}&\\ &\theta(H)\end{smallmatrix}\right] +\left[\begin{smallmatrix} 0&\\ &D_1\end{smallmatrix}\right],
$$
where $H\in\hg$ is defined by $\la H,Y\ra=\tr{\theta(Y)}$ for all $Y\in\ug$ (it easily follows that $[\kg,H]=0$).  We have that $D\in\Der(\ggo)$ since $D_1\in\Der(\ngo)$ (see (d1)) and $[\theta(\ug),D_1]=0$ (see Remark \ref{Mcomm}).

It follows from the converse part of Theorem \ref{main} that the Ricci operator of the reductive decomposition $\gkp$ is given by
\begin{equation}\label{cons-ric}
\Ricci=cI+S(D_{\pg}) = cI+\left[\begin{smallmatrix} -S(\ad_\ug{H}|_\hg)&\\ &-S(\theta(H))+D_1\end{smallmatrix}\right].
\end{equation}
We note that $\Ricci=cI$ if and only if $S(\ad_\ug{H})=0$ and $D_1=\unm (\theta(H)+\theta(H)^t)$.  On the other hand, it is an algebraic soliton if and only if $S(\ad_\ug{H}|_\hg)=0$, if and only if $\ad_\ug{H}$ and $\theta(H)$ are both normal operators (see Proposition \ref{equiv-algsol}).

We now rephrase the main part of Theorem \ref{main} in simpler terms.

\begin{theorem}\label{main2}
Any homogeneous expanding Ricci soliton is isometric to a homogeneous space $(G/K,g)$ with a metric reductive decomposition constructed as above.
\end{theorem}

\begin{proof}
According to Theorem \ref{semi}, any homogeneous Ricci soliton is isometric to a semi-algebraic soliton $(G/K,g)$.  Now consider the metric reductive decomposition such that $B(\kg,\pg)=0$.  It follows from Theorem \ref{main}, parts (i)-(iv), that such decomposition is obtained by the construction procedure described above with $\theta(Y):=\ad{Y}|_{\ngo}$ for all $Y\in\ug$, concluding the proof.
\end{proof}

Let $(G/K,g)$ be any connected homogeneous space with metric reductive decomposition $\gkp$ constructed as above, with $\ggo=\ug\oplus\ngo$.  Assume that $G$ is connected and take $U,N\subset G$, the connected Lie subgroups with Lie algebras $\ug$ and $\ngo$, respectively.  Note that if $K$ is connected then $K\subset U$.

We have that $UN$ is a subgroup of $G$ as $N$ is normal, and since it contains a neighborhood of the identity (recall that the map $\ggo=\ug\oplus\ngo\longrightarrow G$, $(Y,X)\mapsto\exp(Y)\exp(X)$ is a diffeomorphism between some open neighborhoods of $(0,0)$ and $e\in G$, respectively), it follows that $G=UN$.  Now the function
$$
q:U\ltimes N\longrightarrow G, \qquad q(u,n):=un,
$$
is an epimorphism of Lie groups whose derivative is the isomorphism of Lie algebras given by $dq|_e(Y,X)=Y+X$, for all $Y\in\ug$, $X\in\ngo$.  Recall that on the external semi-direct product $U\ltimes N$ the multiplication is defined by $(u,n)\cdot(v,m)=(uv,v^{-1}nvm)$.  Thus $q$ is a covering map and its kernel is a discrete subgroup of the center of $U\ltimes N$, which is precisely given by the anti-diagonal
$$
\Ker(q)=\Delta(U\cap N):=\{ (x,x^{-1}):x\in U\cap N\}.
$$
Moreover, we obtain that $G$ is isomorphic to $(U\ltimes N)/\Delta(U\cap N)$.  In particular, if $G$ is simply connected, then $G\simeq U\ltimes N$, and so $G$ is diffeomorphic to the direct product $U\times N$, from which follows that $U$ and $N$ are both simply connected as well.

Let us state some of these properties for future use.

\begin{proposition}\label{GKtop}
Let $(G/K,g)$ be a homogeneous space with metric reductive decomposition $\gkp$ constructed as above and assume that $G$ is simply connected and $K$ connected.  Then $G\simeq U\ltimes N$, the groups $U$ and $N$ are simply connected and $G/K$ is diffeomorphic to the product of manifolds $U/K\times N$.
\end{proposition}

It follows from Theorem \ref{main2} that a given simply connected Einstein $(G/K,g)$ obeys the Alekseevskii's conjecture if and only if $U/K$ is diffeomorphic to a Euclidean space, or equivalently, $K$ is a maximal compact subgroup of $U$ (see Section \ref{ERS}).

\section{Algebraic solitons and the Alekseevskii Conjecture}\label{algsol-alek}

As another application of the given structural results, in this section we prove a close link between Einstein homogeneous manifolds and algebraic solitons. We obtain as a consequence that if the Alekseevskii conjecture turns out to be true, then the following a priori much stronger result also holds:

\begin{quote}
Any expanding algebraic soliton $(G/K,g)$ is diffeomorphic to a Euclidean space (or equivalently, $K$ is a maximal compact subgroup of $G$).
\end{quote}

Assume that we have a metric reductive decomposition  $(\ggo = \kg \oplus \pg, \ip)$ such that $\ggo$ is non-unimodular.  Thus the vector $H\in\pg$ defined in \eqref{defH} is nonzero, and since $[\ggo,\ggo]\perp H$, it is clear that the subspace $\ggo_0 := \{ H\}^\perp$ is in fact an ideal of codimension one in $\ggo$. This implies that
$$
\ggo=\RR H\oplus\ggo_0,
$$
is a semidirect product of Lie algebras, with $\ggo_0$ a unimodular ideal.

In the following result we use the previous observation to prove that one can get an Einstein metric reductive decomposition out of any algebraic soliton, either by changing the adjoint action of $H$ on $\ggo_0$, or by adding a suitable one in the unimodular case.

\begin{proposition}\label{Et}
Let $\gkp$ be a metric reductive decomposition such that $B(\kg,\pg)=0$, and assume that it is an algebraic soliton with $\Ricci = cI + D_\pg$, $D=S(D)\in \Der(\ggo)$, as in \eqref{cons-ric}.

\begin{itemize}
\item[(i)]  If $\ggo$ is non-unimodular, consider the new Lie algebra $\wt{\ggo}$ with underlying vector space $\ggo=\RR H\oplus\ggo_0$ and Lie bracket defined by keeping the one of $\ggo$ on $\ggo_0$ and only replacing the adjoint action of $H$ on $\ggo_0$ by
$$
\ad_{\tilde{\ggo}}{H} := \alpha (S(\ad_\ggo{H}) + D) =\alpha \left[\begin{smallmatrix} 0&&\\ &0&\\ &&D_1\end{smallmatrix}\right], \qquad
\alpha=\frac{\| H\|}{(\tr{D_1})^{1/2}}.
$$
Then the metric reductive decomposition $(\wt{\ggo}=\kg\oplus\pg,\ip)$ is Einstein with $\wt{\Ricci} = cI$.


\item[(ii)] The metric reductive decomposition $(\ggo_0=\kg\oplus\pg_0,\ip|_{\ggo_0\times\ggo_0})$, where $\pg_0:=\pg\cap\{ H\}^{\perp}$, is also an algebraic soliton with
$$
\Ricci_{\ggo_0}=cI+D'_{\pg_0}, \qquad \mbox{where}\quad D':= D|_{\ggo_0}+S(\ad{H})|_{\ggo_0}\in\Der(\ggo_0).
$$

\item[(iii)] If $\ggo$ is unimodular, consider the semi-direct product $\wt{\ggo}=\RR A\oplus\ggo$, with
$$
\ad_{\tilde{\ggo}}{A} := \alpha D= \alpha \left[\begin{smallmatrix} 0&&\\ &0&\\ &&D_1\end{smallmatrix}\right], \qquad
\alpha=\frac{1}{(\tr{D_1})^{1/2}}.
$$
Then the metric reductive decomposition $(\wt{\ggo}=\kg\oplus\wt{\pg},\ip)$, where $\wt{\pg}:=\RR A\oplus\pg$ and $\| A\|=1$, $\la A,\pg\ra=0$, is Einstein with $\wt{\Ricci} = cI$.
\end{itemize}
\end{proposition}

\begin{remark}
Parts (i) and (ii) may be viewed as a generalization of the link between Einstein solvmanifolds and nilsolitons discovered in \cite{soliton}.
\end{remark}

\begin{proof}
We first prove parts (i) and (ii).  Let $\wt{\Ricci}$,  $\wt{M}$, $\wt{B}_\pg$ and $\widetilde{H}\in\RR H$ denote the corresponding tensors for the metric reductive decomposition $\tgkp$, as defined in \eqref{defH}-\eqref{R}.  Also, set $\pg_0 := \{Y\in \pg : \la Y,H\ra = 0\}$, so that
$$
\pg = \RR H \oplus \pg_0
$$
is an orthogonal decomposition.  By applying Lemma \ref{calculoR} with the above decomposition (see Remark \ref{remcalculoR} and use that $H\perp [\ggo,\ggo]$) to the metric reductive decompositions $\gkp$ and $\tgkp$, we obtain that
$$
\begin{array}{lll}
\la \mm X,X\ra=\la \mm_0  X,X\ra, &\quad& \la \wt{\mm} X,X\ra=\la \mm_0  X,X\ra,  \\ \\
\la \mm X,H\ra= -\unm\tr{\ad_{\ggo} X (\ad_{{\ggo}} H)^t}, &\quad& \la \wt{\mm} X,H\ra= -\unm\tr{\ad_{\wt{\ggo}} X \ad_{\wt{\ggo}} H}, \\ \\
\la \mm H,H\ra=-\unm\tr{\ad_{{\ggo}} H (\ad_{{\ggo}} H)^t}, &\quad& \la \wt{\mm} H,H\ra=-\unm\tr{\ad_{\wt{\ggo}} H \ad_{\wt{\ggo}} H},
\end{array}
$$
for all $X\in \pg_0$, where $\mm_0 $ is the moment map operator corresponding to the metric reductive decomposition $(\ggo_0 = \kg\oplus \pg_0, \ip|_{\ggo_0\times\ggo_0})$ (and recall from Proposition \ref{equiv-algsol} that $\ad_\ggo H$ and $\ad_{\wt{\ggo}} H$ are both normal operators). This implies that  the Ricci operator of $(\ggo_0=\kg\oplus\pg_0,\ip|_{\ggo_0\times\ggo_0})$ is given by $\Ricci_{\ggo_0}=cI+D|_{\pg_0}+S(\ad_\pg{H}|_{\pg_0})$, and so part (ii) follows.

It is also clear that for $X, X' \in \pg_0$,
\[
\la (\wt{\mm}-\unm \wt{B}_\pg)X,X'\ra = \la (\mm-\unm B_\pg)X, X' \ra = \la (cI + D_{\pg} + S(\ad_\pg H)) X, X' \ra.
\]
We know that $\wt{H} = \beta H$, for some $\beta >0$ which is determined by
\begin{equation}\label{defHtilde}
\alpha\beta(\| H\|^2+\tr{D})= \tr{\ad_{\wt{\ggo}} \wt{H}} = \| \wt{H}\|^2 = \beta^2 \| H \|^2,
\end{equation}
so the formula for $\wt{\Ricci}$ on $\pg_0$ is given by
\begin{align*}
\la \wt{\Ricci} X,X'\ra =& \la \big( cI + D_\pg + S(\ad_\pg{H}) - \wt{\ad}_\pg{\wt H} \big) X, X'\ra \\
=& \la \big( cI + D_\pg + S(\ad_\pg{H}) - \beta \alpha (S(\ad_\pg{H})+D_p)\big) X, X'\ra \\
=& \la \big(cI + \left(1-\alpha \beta\right)\left(D_\pg + S(\ad_\pg H)\right) \big) X, X' \ra.
\end{align*}
By taking $\alpha = \beta^{-1}$, which together with \eqref{defHtilde} yield the formula for $\alpha$ given in the theorem, we obtain that $\wt{\Ricci} = cI$ on $\pg_0$.

On the other hand, it follows from Remarks \ref{Rort} and \ref{Bort} and $\tr{\ad_{\pg}{X}}=\tr{\ad{X}}=0$ that
\begin{align*}
\la \wt{\Ricci} X, H\ra =& \la \wt{\mm} - \unm \wt{B}_\pg) X, H\ra = -\tr{\ad_{\wt{\ggo}}{X} S(\ad_{\wt{\ggo}}{H})}  \\
=& -\alpha\tr{\ad_{\pg}{X} (S(\ad_{\pg}{H})+D_\pg)} \\
=& -\alpha\tr{\ad_{\pg}{X} (-cI+M-\unm B)} =0.
\end{align*}

Finally, by using \eqref{c} we obtain that
\begin{align*}
\la \wt{\Ricci} H, H\ra =& \la (\wt{M} - \unm \wt{B}_\pg) H, H\ra = -\tr{(\wt{\ad}_{\pg}{H})^2}  \\
 =& -\alpha^2\tr{(S(\ad_\pg{H})+D_\pg)^2} =c\alpha^2\tr{S(\ad_\pg{H})+D_\pg}=    c \|H \|^2,
\end{align*}
concluding the proof of part (i).

We now prove part (iii) in a very similar way.  It is easy to see that $\wt{H}=(\tr{D_1})^{1/2}A$, and since
$$
\tr{\ad{A}\ad{X}}=\alpha\tr{D_\pg\ad_\pg{X}}= \alpha\tr{(M-\unm B-cI)\ad_{\pg}{X}}=0, \qquad\forall X\in\pg,
$$
by unimodularity and Remarks \ref{Rort} and \ref{Bort}, we have that
$$
\wt{B}_{\wt{\pg}}=\left[\begin{smallmatrix} (\tr{D_1})^{-1}\tr{D_1^2}&0\\ 0&B_\pg\end{smallmatrix}\right].
$$
By applying Lemma \ref{calculoR} to the decomposition $\wt{\pg}=\RR A\oplus\pg$ (see Remark \ref{remcalculoR}), we obtain that
$$
\la \wt{\mm} X,X\ra=\la \mm  X,X\ra, \quad
\la \wt{\mm} X,A\ra= 0, \quad
\la \wt{\mm} A,A\ra=-\unm(\tr{D_1})^{-1}\tr{D_1^2},
$$
for all $X\in \pg$.  Finally, we use $\Ricci_\ngo=cI+D_1$ to get $c=-(\tr{D_1})^{-1}\tr{D_1^2}$ (see Remark \ref{Rort}), from which it easily follows that $\wt{\Ricci}=\wt{M}-\unm\wt{B}_{\wt{\pg}}-S(\ad_{\wt{\pg}}{\wt{H}})= cI$, concluding the proof of the proposition.
\end{proof}

\begin{theorem}\label{equivconj}
Assume there exists an expanding algebraic soliton which is not diffeomorphic to a Euclidean space.  Then there is a counterexample to the Alekseevskii conjecture.
\end{theorem}

\begin{remark}
This result was proved in \cite{HePtrWyl} (see Remark 1.14 in that paper) in the case of left-invariant metrics on Lie groups by using entirely different methods.  They study warped product structures of Einstein manifolds.  The preprint \cite{HePtrWyl} has been replaced in arXiv by \cite{HePtrWyl2}, where a proof for the homogeneous case is included.
\end{remark}

\begin{proof}
We know that such a soliton is isometric to an algebraic soliton $(G/K,g)$ with a metric reductive decomposition $\gkp$, $\ggo=\ug\oplus\ngo$, constructed as in Section \ref{cons} and with Ricci operator given as in \eqref{cons-ric}, with $D=S(D)$.  We can also assume that $G$ is simply connected and $K$ connected (see Remark \ref{GscKc}).  Consider the Einstein metric reductive decomposition $(\wt{\ggo}=\kg\oplus\wt{\pg},\ip)$ provided by Proposition \ref{Et} out of $(G/K,g)$, and take the corresponding homogeneous space $(\wt{G}/\wt{K},\wt{g})$, where $\wt{G}$ is the simply connected Lie group with Lie algebra $\wt{\ggo}$ and $\wt{K}$ the connected Lie subgroup of $\wt{G}$ with Lie algebra $\kg$.  The fact that $\wt{K}$ is closed in $\wt{G}$ will follow from the analysis below, and note that $\wt{G}/\wt{K}$ is almost-effective since the isotropy $\wt{\kg}$-representation is faithful (by the almost-effectiveness of $G/K$).

We first assume that $\ggo$ is non-unimodular.  It follows from Proposition \ref{GKtop} that $G/K$ and $\wt{G}/\wt{K}$ are respectively diffeomorphic to
$$
U/K \times N, \qquad \wt{U}/\wt{K}\times \wt{N},
$$
where $U$ is the connected Lie subgroup of $G$ with Lie algebra $\ug$ (and analogously for $\wt{U}$).  Now if $U_0$ is the connected Lie subgroup of $G$ with Lie algebra $\ug_0:=\kg\oplus(\hg\cap\ggo_0)$ (and analogously for $\wt{U}_0$), then they are respectively diffeomorphic to
$$
\RR\times U_0/K \times N, \qquad \RR\times \wt{U}_0/\wt{K}\times \wt{N}.
$$
But $U_0\simeq\wt{U}_0$ and $N\simeq\wt{N}$ as they are simply connected and have identical Lie algebras, and since $K$ and $\wt{K}$ are connected, we obtain that $U_0/K$ is diffeomorphic to $\wt{U}_0/\wt{K}$.  Thus $G/K$ and $\wt{G}/\wt{K}$ are diffeomorphic, as $N$ and $\wt{N}$ are nilpotent and hence diffeomorphic to a Euclidean space.  This implies that $\wt{G}/\wt{K}$ is not diffeomorphic to a Euclidean space either, giving a counterexample to the Alekseevskii conjecture since $\wt{\Ricci}=cI$, $c<0$, as was to be shown.

The case when $\ggo$ is unimodular can be proved in much the same way as above.  One obtains here that $\wt{G}/\wt{K}$ is diffeomorphic to $\RR\times G/K$.  This concludes the proof of the theorem.
\end{proof}

\section{Appendix: Moment map stratification for the variety of Lie algebras}\label{git}

We refer to the survey \cite[Sections 3 and 7]{cruzchica} for a more detailed exposition on this subject.

Let us consider the space of all skew-symmetric algebras of dimension $n$, which is
parameterized by the vector space
\begin{align*}
V= & \lam \\ = & \{\mu:\RR^n\times\RR^n\longrightarrow\RR^n : \mu\; \mbox{bilinear and
skew-symmetric}\}.
\end{align*}
Then
$$
\nca=\{\mu\in V:\mu\;\mbox{satisfies Jacobi and is nilpotent}\}
$$
is an algebraic subset of $V$ as the Jacobi identity and the nilpotency condition
can both be written as zeroes of polynomial functions.  $\nca$ is often called the
{\it variety of nilpotent Lie algebras} (of dimension $n$).

There is a natural
linear action of $\G$ on $V$ given by
\begin{equation}\label{action}
h.\mu(X,Y)=h\mu(h^{-1}X,h^{-1}Y), \qquad X,Y\in\RR^n, \quad h\in\G,\quad \mu\in V.
\end{equation}
Recall that $\nca$ is $\G$-invariant and the Lie algebra isomorphism classes are
precisely the $\G$-orbits.  The representation of $\g$ on $V$ obtained by differentiation of
(\ref{action}) is given by
\begin{equation}\label{actiong}
\pi(\alpha)\mu=\alpha\mu(\cdot,\cdot)-\mu(\alpha\cdot,\cdot)-\mu(\cdot,\alpha\cdot),
\qquad \alpha\in\g,\quad\mu\in V.
\end{equation}
We note that $\pi(\alpha)\mu=0$ if and only if $\alpha\in\Der(\mu)$, the Lie algebra
of derivations of the algebra $\mu$.  The canonical inner product $\ip$ on $\RR^n$
determines an $\Or(n)$-invariant inner product on $V$, also denoted by $\ip$, as
follows:
\begin{equation}\label{innV}
\la\mu,\lambda\ra= \sum\la\mu(e_i,e_j),\lambda(e_i,e_j)\ra,
\end{equation}
and also the standard $\Ad(\Or(n))$-invariant inner product on $\g$ given by
\begin{equation}\label{inng}
\la \alpha,\beta\ra=\tr{\alpha \beta^{\mathrm t}}=\sum\la\alpha e_i,\beta e_i\ra,
 \qquad \alpha,\beta\in\g,
\end{equation}
where $\{ e_1,...,e_n\}$ denotes the canonical basis of $\RR^n$.  We note that
$\pi(\alpha)^t=\pi(\alpha^t)$ and $(\ad{\alpha})^t=\ad{\alpha^t}$ for any
$\alpha\in\g$, due to the choice of these canonical inner products everywhere.

We can use $\g=\sog(n)\oplus\sym(n)$ as a Cartan decomposition, where $\sog(n)$ and $\sym(n)$ denote the subspaces of skew-symmetric
and symmetric matrices, respectively.  It is proved in \cite[Proposition
3.5]{minimal} that the moment map $m:V\setminus\{ 0\}\longrightarrow\sym(n)$ for the action
(\ref{action}), which is defined by $\la m(\mu),\alpha\ra=\tfrac{1}{||\mu||^2}\la\pi(\alpha)\mu,\mu\ra$, for all $\alpha\in\g$, $\mu\in V$, is given by
\begin{equation}\label{mmv}
\la m(\mu)X,X\ra=\tfrac{1}{||\mu||^2}\left(-2\sum\la\mu(X,e_i),e_j\ra^2
+\sum\la\mu(e_i,e_j),X\ra^2\right),
\end{equation}
for all $X\in\RR^n$.

Let $\tg$ denote the set of all diagonal $n\times n$ matrices.  If $\{
e_1',...,e_n'\}$ is the basis of $(\RR^n)^*$ dual to the canonical basis $\{
e_1,...,e_n\}$, then
$$
\{ v_{ijk}=(e_i'\wedge e_j')\otimes e_k : 1\leq i<j\leq n, \; 1\leq k\leq n\}
$$
is a basis of weight vectors of $V$ for the action (\ref{action}), where $v_{ijk}$
is actually the bilinear form on $\RR^n$ defined by
$v_{ijk}(e_i,e_j)=-v_{ijk}(e_j,e_i)=e_k$ and zero otherwise.  The corresponding
weights $\alpha_{ij}^k\in\tg$, $i<j$, are given by
\begin{equation}\label{alfas}
\pi(\alpha)v_{ijk}=(a_k-a_i-a_j)v_{ijk}=\la\alpha,\alpha_{ij}^k\ra v_{ijk},
\quad\forall\alpha=\left[\begin{smallmatrix} a_1&&\\ &\ddots&\\ &&a_n
\end{smallmatrix}\right]\in\tg,
\end{equation}
where $\alpha_{ij}^k=E_{kk}-E_{ii}-E_{jj}$ and $\ip$ is the inner product defined in
(\ref{inng}).  As usual $E_{rs}$ denotes the matrix whose only nonzero coefficient
is $1$ at entry $rs$.  Let us denote by $\mu_{ij}^k$ the structure constants of a
vector $\mu\in V$ with respect to the basis $\{ v_{ijk}\}$:
$$
\mu=\sum\mu_{ij}^kv_{ijk}, \qquad \mu_{ij}^k\in\RR, \qquad {\rm i.e.}\quad
\mu(e_i,e_j)=\sum_{k=1}^n\mu_{ij}^ke_k, \quad i<j.
$$
Each nonzero $\mu\in V$ uniquely determines an element $\beta_{\mu}\in\tg$ given by
$$
\beta_{\mu}:=\mcc\left\{\alpha_{ij}^k:\mu_{ij}^k\ne 0\right\},
$$
where $\mcc(X)$ denotes the unique element of minimal norm in the convex hull
$\CH(X)$ of a subset $X\subset\tg$.  We note that $\beta_{\mu}$ is always nonzero
since $\tr{\alpha_{ij}^k}=-1$ for all $i<j$ and consequently $\tr{\beta_{\mu}}=-1$.

Let $\tg^+$ denote the Weyl chamber of $\g$ given by
\begin{equation}\label{weyl}
\tg^+=\left\{\left[\begin{smallmatrix} a_1&&\\ &\ddots&\\ &&a_n
\end{smallmatrix}\right]\in\tg:a_1\leq...\leq a_n\right\}.
\end{equation}

In \cite{standard}, a $\G$-invariant stratification for $V=\lam$ has been defined by
adapting to this context the construction given in \cite[Section 12]{Krw1} for
reductive group representations over an algebraically closed field.  We summarize in the following
theorem the main properties of the stratification, which has provided one of the main tools to study the structure of homogeneous Ricci solitons in this paper.

\begin{theorem}\label{strata}\cite{standard, einsteinsolv}
There exists a finite subset $\bca\subset\tg^+$, and for each $\beta\in\bca$ a
$\G$-invariant subset $\sca_{\beta}\subset V$ (a {\it stratum}) such that
$$
V\smallsetminus\{ 0\}=\bigcup_{\beta\in\bca}\sca_{\beta} \qquad \mbox{(disjoint
union)},
$$
and $\tr{\beta}=-1$ for any $\beta\in\bca$.  For $\mu\in\sca_{\beta}$ we have that
\begin{equation}\label{adbeta}
\left\la[\beta,D],D\right\ra\geq 0 \qquad\forall\; D\in\Der(\mu)
\qquad(\mbox{equality holds}\;\Leftrightarrow [\beta,D]=0),
\end{equation}
\begin{equation}\label{betapos}
\beta+||\beta||^2I \quad\mbox{is positive definite for all}\; \beta\in\bca
\;\mbox{such that}\; \sca_{\beta}\cap\nca\ne\emptyset,\;\mbox{and}
\end{equation}
\begin{equation}\label{bmu}
||\beta||\leq ||m(\mu)||\qquad(\mbox{equality holds}\;\Leftrightarrow
m(\mu)\;\mbox{is conjugate to}\; \beta).
\end{equation}
If in addition, $\mu\in\sca_{\beta}$ satisfies $\beta_{\mu}=\beta$, or equivalently,
$$
\min\left\{\la\beta,\alpha_{ij}^k\ra:\mu_{ij}^k\ne 0\right\}=||\beta||^2,
$$
which always holds for some $g.\mu$, $g\in\Or(n)$, then
\begin{equation}\label{betaort}
\tr{\beta D}=0 \quad\forall\; D\in\Der(\mu), \;\mbox{and}
\end{equation}
\begin{equation}\label{delta}
\left\la\pi\left(\beta+||\beta||^2I\right)\mu,\mu\right\ra\geq 0,
\end{equation}
where equality holds if and only if $\beta+||\beta||^2I\in\Der(\mu)$.
\end{theorem}

This stratification is based on instability results and is strongly related to the
moment map in many ways other than (\ref{bmu})  (see \cite{cruzchica}).

\end{document}